\documentclass{amsart}
\usepackage{amsmath}
\usepackage{amsthm}
\usepackage{amssymb}
\usepackage{dsfont}
\usepackage{graphicx}
\usepackage{caption}
\usepackage{subcaption}
\usepackage{color}
\usepackage[english]{babel}
\usepackage{layout}

\newtheorem{definition}{Definition}[section]
\newtheorem{lemma}[definition]{Lemma}
\newtheorem{theorem}[definition]{Theorem}
\newtheorem{proposition}[definition]{Proposition}

\newtheorem{remark}[definition]{Remark}

\newcommand{\dx}{\,dx}

\newcommand{\ie}{; {\it i.e., }}
\newcommand{\e}{\varepsilon}

\newcommand{\lu}{{L}^1}
\newcommand{\lud}{\lu(D)}

\newcommand{\dist}{{\rm dist}\,}

\newcommand{\Om}{\Omega}

\font\tenmsb=msbm10
\font\sevenmsb=msbm7
\font\fivemsb=msbm5
\newfam\msbfam
\textfont\msbfam=\tenmsb
\scriptfont\msbfam=\sevenmsb
\scriptscriptfont\msbfam=\fivemsb

\def\R{\mathbb{R}}

\newcommand{\rn}{\R^n}

\newcommand{\NN}{{\mathbb N}}

\newcommand{\ZZ}{{\mathbb Z}}
\newcommand{\Zn}{{\mathbb{Z}}^n}

\def\to{\rightarrow}

\def\w{\omega}
\def\Lw{\mathcal{L}(\w)}
\def\NNw{\mathcal{NN}(\w)}
\def\Ard{\mathcal{A}^R(D)}
\def\a{\alpha}
\def\rzn{r^{\prime}\mathbb{Z}^n}
\def\NNS{\mathcal{NN}(\Sigma)}

\begin{document}

\author{Roberto Alicandro}
\address[Roberto Alicandro]{DIEI, Universit\`a  di Cassino e del Lazio meridionale, via Di Biasio 43, 03043 Cassino (FR), Italy}
\email{alicandr@unicas.it}

\author{Marco Cicalese}
\address[Marco Cicalese]{Zentrum Mathematik - M7, Technische Universit\"at M\"unchen, Boltzmannstrasse 3, 85748 Garching, Germany}
\email{cicalese@ma.tum.de}

\author{Matthias Ruf}
\address[Matthias Ruf]{Zentrum Mathematik - M7, Technische Universit\"at M\"unchen, Boltzmannstrasse 3, 85748 Garching, Germany}
\email{mruf@ma.tum.de}

\title{Domain formation in magnetic polymer composites: an approach via stochastic homogenization}

\begin{abstract}
We study the magnetic energy of magnetic polymer composite materials as the average distance between magnetic particles vanishes. 
We model the position of these particles in the polymeric matrix as a stochastic lattice scaled by a small parameter $\e$ and the magnets as classical $\pm 1$ spin variables interacting via an Ising type energy. Under surface scaling of the energy we prove, in terms of $\Gamma$-convergence that, up to subsequences, the (continuum) $\Gamma$-limit of these energies is finite on the set of Caccioppoli partitions representing the magnetic Weiss domains where it has a local integral structure. Assuming stationarity of the stochastic lattice, we can make use of ergodic theory to further show that the $\Gamma$-limit exists and that the integrand is given by an asymptotic homogenization formula which becomes deterministic if the lattice is ergodic.
\end{abstract}

\maketitle

\tableofcontents
\section{Introduction}
Magnetic polymer composite materials have raised the attention of the scientific community in the last decades mainly because of their biomedical applications. 
These materials, synthesized by embedding magnetic particles into a polymer matrix, have light weight and high shape-flexibility and are commonly used for bio-magnetic separations processes.

\medskip

In this paper we start the rigorous mathematical study of the discrete-to-continuum variational description of these materials focusing on their magnetic properties. Our aim is to prove rigorously that, modeling magnetic particles as classical Ising spins (see e.g. \cite{Presutti}) sitting on a disordered lattice, their (surface scaled) microscopic interaction energy leads to the formation of Weiss domains as the average distance between the particles vanishes. As explained below in this introduction, in order to tackle this problem we regard it as a stochastic homogenization problem in the space of functions of bounded variation in $\R^{n}$ where we are able to extend some of the results obtained in the Sobolev setting in the pioneering paper \cite{DMM}. \\

We recall here that the variational analysis of the properties of ground states of Ising-like systems can be traced back to the pioneering paper by Caffarelli and de la Llave \cite{CdlL} and that derivation of continuum theories from atomistic spin-like ones in the framework of $\Gamma$-convergence is not new and has been initiated by Braides and collaborators in \cite{ABC}. Since then it has been developed by many authors in connection with the theory of surfactants, nematic elastomers, dislocations in plasticity, superfluids or frustrated magnetic chains, to cite a few (see for instance \cite{ACXY, ACP, ACS, BCS, BP, BrSo, CS, Ponsiglione}). \\

The modeling of magnetic polymer composite materials at a small (micro or nano) scales requires the modeling of two main objects: a polymer matrix and an interaction energy between the magnetic particles (see \cite{vollath2013} and reference therein for a beginner's guide to this topics).  \\

{\bf The polymer matrix}

The polymer matrix can be modeled as a random network having the cross-linked molecules as nodes. We will suppose the nodes of the network to satisfy some minimal geometric assumption uniformly in the randomness. More precisely we will suppose the set of the nodes of the network to form what we call an {\sl admissible stochastic lattice} according to the definition below. Note that this definition can be considered standard in the framework of statistical mechanics (see for instance \cite{ruelle}) and it was first used in the context of atomic-to-continuum limit in \cite{BLBL} as well as in \cite{ACG2} in the analysis of rubber elasticity models. Given a countable set of points $\Sigma=\{x_i\}_{i\in\NN}$ in $\R^n$, we say that $\Sigma$ is admissible if
\begin{itemize}
\item[(i)] there exists $R>0$ such that $\inf_{z\in\R^n}\#(\Sigma\cap B(z,R))\geq 1$ (i.e., arbitrarily big empty regions are forbidden),
\item[(ii)] there exists $r>0$ such that $\inf\{|x-y|,\ x,y\in\Sigma,\ \ x\not=y\}\geq r$ (i.e., clusters are forbidden).
\end{itemize}
Then, given a probability space $(\Omega,{\mathcal F},\mathbb P)$, a random variable ${\mathcal L}:\Omega\to(\R^n)^{\NN}$ is called an {\sl admissible stochastic lattice} if, uniformly with respect to $\omega\in\Omega$, ${\mathcal L}(\omega)$ is an admissible set of points. Note that our assumptions on the admissibility of a stochastic lattice rule out many point processes well known in probability theory and are instead motivated by the usual structural assumptions on the polymeric matrix. \\

{\bf The atomic energy}

\noindent To every stochastic lattice ${\mathcal L}(\omega)$ we associate a Voronoi tessellation ${\mathcal V}({\mathcal L}(\omega))$ and define the set of nearest neighboring points, namely $\mathcal {NN}(\omega)$, as the set of those pairs of points of the stochastic lattice ${\mathcal L}(\omega)$ which share a $(n-1)$-dimensional edge of the associated Voronoi tessellation.
Let $D\subset \R^n$ be a bounded open set, and $\e>0$ be a small parameter (the limit $\e\to 0$ will be referred to as the continuum limit). We assume that the magnetic state of the particles in $D$ is described by a classical spin variable $u:\e{\mathcal L}(\omega)\cap D\to\{\pm 1\}$ and we model the interactions between the spins via an Ising type energy. The energy model we consider allows all the particles to interact and may distinguish between short-range interactions, which are the interactions between the nearest-neighbors particles, and long-range interactions. The total energy of the system for a given configuration $u$ has the form 
$$F_{\e}(\w)(u):=F_{nn,\e}(\w)(u,D)+F_{lr,\e}(\w)(u,D),$$
where
\begin{align*}
F_{nn,\e}(\w)(u,D)&=\sum_{\substack{(x,y)\in\NNw \\ \e x,\e y\in D}}\e^{n-1} c_{nn}^{\e}(x,y)|u(\e x)-u(\e y)|,\\
F_{lr,\e}(\w)(u,D)&=\sum_{\substack{(x,y)\notin\NNw \\ \e x,\e y\in D}}\e^{n-1} c_{lr}^{\e}(x,y)|u(\e x)-u(\e y)|.
\end{align*}
For $c_{nn}^{\e},c_{lr}^{\e}:\rn\times\rn\rightarrow [0,+\infty)$ we assume that there exist $C>0$ and a decreasing function $J_{lr}:[0,+\infty)\rightarrow[0,+\infty)$ with 
\begin{equation*}
\int_{\rn}J_{lr}(|x|)|x|\,\mathrm{d}x =J<+\infty
\end{equation*}
such that, for all $\varepsilon>0$ and all $x,y\in\rn$,
\begin{align*}
&\frac{1}{C}\leq c_{nn}^{\e}(x,y)\leq C,  \\
&c_{lr}^{\e}(x,y)\leq J_{lr}(|x-y|). 
\end{align*}
As the average distance between the nodes of the network $\e\Lw$ is of order $\e$, the prefactor $\e^{n-1}$ in the energy has the meaning of a surface scaling, so that $F_{\e}(\w)(u)$ is the magnetic energy per unit surface of the network $\e\Lw\cap D$ when the magnetization field is $u$. Taking into account the assumptions above, the atomic system we consider is the surface scaling of a ferromagnetic type system with bounded short-range and summable long-range interactions.\\ 

\smallskip

{\bf The continuum energy}

In the limit as $\e$ tends to $0$ the ferromagnetic behavior of the system will favor the formation of a partition of $D$ into random ($\w$ dependent) Weiss domains described, in the continuum limit, as sets of finite perimeter with fixed magnetization $+1$ or $-1$. The interaction energy between the Weiss domains will depend on the randomness via the stochasticity of the polymer matrix in which the magnetic particles are embedded. The issue of the dependence of the macroscopic continuum energy of the domains on the randomness of the matrix is tackled in the framework of stochastic homogenization as explained below. In this context, as a byproduct of our analysis, one could see our main result as a generalization of a recent theorem by Braides and Piatnitski in \cite{BP} (see also Remark \ref{stoc-coef}).\\

We work in the variational framework of $\Gamma$-convergence (we refer to \cite{GCB, DM} for an introduction to the subject). To this end we identify the field $u$ with its piecewise-constant interpolation taking the value $u(x)$ on the Voronoi cell centered at $x$ and we regard the energies as defined on $L^1(D,\{\pm 1\})$. The $\Gamma$-limit is performed in this space. In Theorem \ref{mainthm1}, we prove that, for fixed $\w\in\Omega$, up to subsequences, the family $F_{\e}(\w)$ $\Gamma$-converges with respect to the $L^{1}(D)$- topology to a continuum energy $F:L^{1}(D)\to [0,+\infty]$ which is finite only on $BV(D,\{\pm 1\})$ where it takes the form
\begin{equation}\label{intro:representation}
F(\omega)(u)=\int_{S(u)\cap D}\phi_{\Lw}(x,\nu_u)\,\mathrm{d}\mathcal{H}^{n-1}. 
\end{equation}
Here $S(u)$ denotes the jump set of $u$, $\nu_{u}\in S^{n-1}$ its measure theoretic inner normal and $\mathcal{H}^{n-1}$ the $(n-1)$-dimensional Hausdorff measure. 
The result is proved by the abstract methods of Gamma-convergence and makes use of the integral representation theorem in \cite{BFLM}. 
We explore the dependence of the continuum energy on the randomness induced by the stochastic lattice in Theorem \ref{mainthm2}. Here we assume that 
the stochastic lattice is stationary, that is, for all $z\in\ZZ^n$, ${\mathcal L}(\omega)$ and ${\mathcal L}(\omega)+z$ have the same statistics and that 
there exist two functions $c_{nn},c_{lr}:\rn\rightarrow [0,+\infty)$ such that
\begin{equation}\label{intro:periodicityassumpt}
c_{nn}^{\e}(x,y)=c_{nn}(y-x),\quad c_{lr}^{\e}(x,y)=c_{lr}(y-x).
\end{equation}
These assumptions, which play the same role as periodicity in the case of a deterministic periodic lattice treated in \cite{AlGe14},
turn the problem of the characterization of the continuum limit energy into a stochastic homogenization problem. 

In Theorem \ref{mainthm2} we prove that the functionals $F_{\e}(\w)$ $\Gamma$-converge with respect to the $\lud$-topology to the functional $F_{\text{hom}}(\w):\lud\rightarrow [0,+\infty]$ which is finite on $BV(D,\{\pm 1\})$ where it takes the form
\begin{equation*}
F_{\text{hom}}(\w)(u)= \int_{S(u)\cap D} \phi_{\text{hom}}(\w;\nu_u)\,\mathrm{d}\mathcal{H}^{n-1} 
\end{equation*}
where, for $\mathbb{P}$-almost every $\w$ and for all $\nu\in S^{n-1}$, $\phi_{\text{hom}}(\w;\nu_u)$ is given by an asymptotic homogenization formula.
In case $\mathcal{L}$ is ergodic the limit energy is deterministic and its energy density $\phi_{\text{hom}}(\nu)$ is obtained by averaging over the probability space: 
\begin{equation*}
\phi_{\text{hom}}(\nu)=\int_{\Om}\phi_{\text{hom}}(\w;\nu_u)\,\mathrm{d}\mathbb{P}(\w).
\end{equation*}
The proof of this result is quite delicate and makes use of two main ingredients: the abstract methods of $\Gamma$-convergence and the subadditive ergodic theorem by Ackoglu and Krengel in \cite{ergodic}. The combination of these two results in the framework of discrete-to-continuum limits was one of the key ideas in the proof of the main result in \cite{ACG2} drawing some ideas from the pioneering paper \cite{DMM}. It consists in proving that the sequence of minimum problems characterizing the energy density of the $\Gamma$-limit at a certain point and in a given direction agrees (up to lower order terms) with a sequence of sub additive stochastic processes for which the main result in \cite{ergodic} applies. It is at this point that one strongly uses the assumptions on the stationarity of the lattice together with \eqref{intro:periodicityassumpt}. This step of the proof is the most delicate one and cannot be solved by the same arguments as in the Sobolev case considered in \cite{ACG2}. Instead it requires new arguments and the generalization to higher dimensions of the translation invariance of the first passage percolation formula (Proposition 2.10) in \cite{BP}. \\

A further important issue in the theory of magnetic polymer composite materials is the dependence of the macroscopic energy on the random geometry of the polymer matrix. We consider this problem in Section \ref{sect:last} where we remark that if the polymer matrix, besides satisfying the previous assumptions, is also isotropic in law, that is to say that $\Lw$ and $R\Lw$ have the same statistics for all $R\in SO(n)$, and the coefficients $c_{nn}$ and $c_{lr}$ are functions of the distances between points, then the limit energy density is isotropic, which means that $\phi_{\text{hom}}(\nu)=const$. In this case the energy of the continuum system is proportional to the length of the boundary of the Weiss domains. An example of stochastic lattice sharing this isotropy in law is the random parking process studied from the point of view of homogenization theory in \cite{glpe}. \\

Last but not least, we remark that the magnetic polymer composite materials present non trivial mechanical response to an applied magnetic field, a phenomenon known as strain-alignment coupling. A variational discrete to continuum analysis of this phenomenon, whose equivalent formulation in the periodic setting has been considered for instance in \cite{CDSZ}, is possible in this framework using some of the ideas contained in this paper and in \cite{ACG2}. We leave this topic for future studies.  \\

The paper is organized as follows:
section \ref{sect:pre} is devoted to basic notation, the definition of the class of energies we consider and to preliminary results regarding our functional setting. In section \ref{sec:intrep} we prove a compactness and integral representation result for our functionals for a fixed realization of the random lattice. Section \ref{sec:bdy} deals with the $\Gamma$-convergence of the discrete energies in presence of boundary conditions while in Section \ref{sec:sto} we prove the main result of this paper that is the $\Gamma$-convergence of random discrete energies. The last section is devoted to applications and generalizations of the previous results to multi-body magnetic interaction potentials and nonlocal energy functionals. 

\section{Notation and preliminaries}\label{sect:pre}
In this section we introduce some notation which we will use in the following and give a precise definition of the energies we consider. 
\\
\hspace*{0,5cm}
By $|\cdot|$ we denote the Euclidean norm on $\rn$. If $A\subset\rn$ is a Borel set we denote by $|A|$ its Lebesgue measure and, for any $\delta>0$, we set $A^{\delta}=A+B_{\delta}(0)$, where $B_{\delta}(x)$ is the open ball around $x$ with radius $\delta$ with respect to $|\cdot|$. Given an open set $D\subset\rn$ we denote by $\mathcal{A}(D)$ the family of all bounded open subsets of $D$ and by $\Ard$ the family of those sets in $\mathcal{A}(D)$ which have a Lipschitz boundary. Moreover, we set $\text{dim}_{\mathcal{H}}(\cdot)$ the Hausdorff dimension. Given a unit vector $\nu_1\in S^{n-1}$, let $\nu_1,\dots,\nu_n$ be a orthonormal basis. We define

\begin{equation*}
Q_{\nu}=\{x\in\rn:\;|\langle x,\nu_i\rangle|\leq\frac{1}{2}\quad\forall i\}
\end{equation*}

and, for $x\in\rn,\rho>0$, we set $Q_{\nu}(x,\rho):=x+\rho\, Q_{\nu}$. In the proofs $C$ denotes a generic constant that can change every time it appears.\\

\medskip

\noindent In this paper we consider as admissible networks those set of points fulfilling the following definition. 
\begin{definition}\label{defadmissible}
Let $\Sigma$ be a countable set of points in $\rn$. $\Sigma$ is called an admissible set of points if the following two conditions hold:
\begin{itemize}
\item[(i)] There exists $R>0$ such that $\inf_{x\in\rn} \#\left(\Sigma\cap B_R(x)\right)\geq 1$;
\item [(ii)] There exists $r>0$ such that $|x-y|\geq r$ for all $x,y\in\Sigma,x\neq y$.
\end{itemize}
\end{definition}
We remark that this class of lattices has been introduced in \cite{BLB, BLBL} (and already used in \cite{ACG2} in the context of $\Gamma$-convergence) as a reference configuration of nonlinear elastic networks. Roughly speaking the assumptions above rule out cluster points as well as arbitrary big holes in the network.
\begin{definition}
Let $\Sigma$ be a countable set of points in $\rn$. We denote by $\mathcal{V}(\Sigma)$ the so called Voronoi tessellation of $\rn$ associated with $\Sigma$, i.e. $\mathcal{V}(\Sigma):=\{C(x)\}_{x\in\Sigma}$, where
\begin{equation*}
C(x):=\{z\in\rn:\;|z-x|\leq|z-y|\;\forall y\in\Sigma\}.
\end{equation*}
\end{definition}

The next lemma contains all the information on the Voronoi cells that we will need throughout the paper. We outline its simple proof for readers' convenience.

\begin{lemma}\label{cellproperty}
Let $\Sigma$ be an admissible set of points with constants $r,R$ as in Definition \ref{defadmissible}. Then there exist constants $M_1,M_2>0$ depending only on $r,R$ such that, for all $x\in\Sigma$,
\begin{itemize}
\item [(i)] $B_{\frac{r}{2}}(x)\subset C(x)\subset B_R(x)$,
\item [(ii)] $\#\{y\in\Sigma:\;C(x)\cap C(y)\neq\emptyset\}\leq M_1$.
\item [(iii)] $\mathcal{H}^{n-1}(C(x)\cap C(y))\leq M_2\quad\forall y\in \Sigma\backslash\{x\}$,
\end{itemize}
\end{lemma}

\begin{proof} (i) For $y\in\Sigma\backslash\{x\}$ we have $|x-y|\geq r$, which implies $|z-x|\leq |z-y|$ for all $z\in B_{\frac{r}{2}}(x)$. By definition the first inclusion in (i) holds. Now suppose that there exists $z\in C(x)$ such that $|z-x|\geq R$. Since $\Sigma$ is admissible, there exists $y\in\Sigma$ such that $|z-y|<R$. It follows that

$$R\leq |z-x|\leq|z-y|<R,$$

leading to a contradiction.\\
(ii) Note that (i) implies that if $C(x)\cap C(y)\neq\emptyset$, then $|x-y|\leq 2R$. Using an elementary covering argument it is now easy to see that it is enough to take $M_1=\left(1+\frac{4R}{r}\right)^n$.\\
(iii) By (i) the diameter of the set $C(x)\cap C(y)$ is bounded by $2R$ and the set is contained in a $(n-1)$-dimensional affine subspace so that we can take $M_2=(2R)^{n-1}\omega_{n-1}$, where $\omega_{n-1}$ is the volume of the unit ball in $\mathbb{R}^{n-1}$.

\end{proof}

Let $D\subset\rn$ be a bounded open set with Lipschitz boundary and let $\Sigma$ be an admissible set of points according to Definition \ref{defadmissible}. Making use of the Voronoi tessellation we introduce the notion of nearest neighbors.

\begin{definition}\label{nn}
The set of nearest neighbors of $\Sigma$ is defined by
$$\NNS:=\{(x,y)\in\Sigma^2:\;\dim_{\mathcal{H}}\left(C(x)\cap C(y)\right)=n-1\}.$$
\end{definition}

\hspace*{0,5cm}
We are now ready to introduce the most general class of discrete energies we are going to consider in this paper. For fixed $\e>0$ and $u:\e\Sigma\rightarrow\{\pm 1\}$, we set
$$F_{\e}(u):=F_{nn,\e}(u,D)+F_{lr,\e}(u,D),$$
where for every $A\in\mathcal{A}(\rn)$
\begin{align}
F_{nn,\e}(u,A)&=\sum_{\substack{(x,y)\in\NNS \\ \e x,\e y\in A}}\e^{n-1} c_{nn}^{\e}(x,y)|u(\e x)-u(\e y)|,\label{Fnn}\\
F_{lr,\e}(u,A)&=\sum_{\substack{(x,y)\notin\NNS \\ \e x,\e y\in A}}\e^{n-1} c_{lr}^{\e}(x,y)|u(\e x)-u(\e y)|.\label{Flr}
\end{align}
The functions $c_{nn}^{\e},c_{lr}^{\e}:\rn\times\rn\rightarrow [0,+\infty)$ fulfill the following assumptions:\\

\medskip

\noindent \textbf{Hypothesis 1 }There exist $C>0$ and a decreasing function $J_{lr}:[0,+\infty)\rightarrow[0,+\infty)$ with 

\begin{equation*}
\int_{\rn}J_{lr}(|x|)|x|\dx =J<+\infty
\end{equation*}

such that, for all $\varepsilon>0$ and all $x,y\in\rn$,
\begin{align*}
&\frac{1}{C}\leq c_{nn}^{\e}(x,y)\leq C,  \\
&c_{lr}^{\e}(x,y)\leq J_{lr}(|x-y|). 
\end{align*}

\hspace*{0,5cm}
As it is customary in the context of the discrete-to-continuum variational limit, with the aim of exploiting $\Gamma$-convergence, we identify each function $u:\e\Sigma\rightarrow\{\pm 1\}$ with its constant interpolation on each Voronoi cell. Setting

\begin{equation}\label{CepsSig}
C_{\e}(\Sigma):=\{u:\rn\rightarrow\{\pm 1\}:\;\forall C\in\mathcal{V}(\Sigma),\,u_{|\e C} \text{ is constant}\}\subset \lud,
\end{equation}
we can consider the functionals $F_{\e}:\lud\rightarrow [0,+\infty]$ defined as

\begin{equation}\label{defL1}
F_{\e}(u):=
\begin{cases}
F_{nn,\e}(u,D)+F_{lr,\e}(u,D) &\mbox{if $u\in C_{\e}(\Sigma)$,}\\
+\infty &\mbox{otherwise.}
\end{cases}
\end{equation}

With the aim of applying the abstract methods of $\Gamma$-convergence we also need to define local versions of the energies $F_{\e}$ and of its $\Gamma\hbox{-}\liminf$ and $\Gamma\hbox{-}\limsup$ as $\e\to 0$.

\begin{definition}\label{localenergy}
For every $A\in\mathcal{A}^R(\rn)$, let $F_{\e}(\cdot,A):L^1(D)\rightarrow [0,+\infty]$ be defined by

\begin{align*}
F_{\e}(u,A)=
\begin{cases}
F_{nn,\e}(u,A)+F_{lr,\e}(u,A) &\mbox{if $u\in C_{\e}(\Sigma)$,}\\
+\infty &\mbox{otherwise,}
\end{cases}
\end{align*}

where $F_{nn,\e}(u,A)$ and $F_{lr,\e}(u,A)$ are defined as in (\ref{Fnn}) and in (\ref{Flr}). Furthermore we set
\begin{align*}
F^{\prime}(u,A)&:=\Gamma(\lud)-\liminf_{\e\to 0}F_{\e}(u,A),\\
F^{\prime\prime}(u,A)&:=\Gamma(\lud)-\limsup_{\e\to 0}F_{\e}(u,A).
\end{align*}
\end{definition}

\begin{remark}\label{localgamma}\rm{
One can show that 

\begin{align*}
F^{\prime}(u,A)&=\Gamma(L^1(A))-\liminf_{\e\to 0}F_{\e}(u,A),\\
F^{\prime\prime}(u,A)&=\Gamma(L^1(A))-\limsup_{\e\to 0}F_{\e}(u,A)
\end{align*}

for every $u\in \lud$.
}
\end{remark} 
Following some ideas in \cite{ACG2} we introduce an auxiliary deterministic square lattice on which we will conveniently rewrite the energies $F_{\e}$. This lattice will turn out to be a convenient way in order to provide uniform (with respect to the stochastic variable) estimates on the discrete energies.\\

On setting  $r^{\prime}=\frac{r}{\sqrt{n}}$ it follows that for all $\alpha\in r^{\prime}\mathbb{Z}^n$ it holds $\#\{\Sigma\cap\{\alpha+[0,r^{\prime})^n\}\}\leq 1$. We now set
\begin{align*}
\mathcal{Z}_{r^{\prime}}(\Sigma):=&\{\a\in r^{\prime}\mathbb{Z}^n:\;\#\left(\Sigma\cap \{\a+[0,r^{\prime})^n\}\right)=1\},
\\
x_{\a}:=& \Sigma\cap\{\a+[0,r^{\prime})^n\},\quad\a\in\mathcal{Z}_{r^{\prime}}(\Sigma)
\end{align*}
and, for $\xi\in r^{\prime}\mathbb{Z}^n,\,U\subset\rn$ and $\e>0$, 
\begin{align*}
R^{\xi}_{n\!n,\e}(U)&:=\{\a:\;\a,\a+\xi\in\mathcal{Z}_{r^{\prime}}(\Sigma),\,\e x_{\a},\e x_{\a+\xi}\in U,\,(x_{\a},x_{\a+\xi})\in \NNS\},
\\
R^{\xi}_{lr,\e}(U)&:=\{\a:\;\a,\a+\xi\in\mathcal{Z}_{r^{\prime}}(\Sigma),\,\e x_{\a},\e x_{\a+\xi}\in U,\,(x_{\a},x_{\a+\xi})\notin \NNS\}.
\end{align*}

We can now rewrite the energy as
 
\begin{align*}
F_{nn,\e}(u,A)&=\sum_{\xi\in \rzn}\sum_{\a\in R^{\xi}_{n\!n,\e}(A)}\e^{n-1}c_{nn}^{\e}(x_{\a},x_{\a+\xi})|u(x_{\a})-u(\e x_{\a+\xi})|,\\
F_{lr,\e}(u,A)&=\sum_{\xi\in \rzn}\sum_{\a\in R^{\xi}_{lr,\e}(A)}\e^{n-1}c_{lr}^{\e}(x_{\a},x_{\a+\xi})|u(\e x_{\a})-u(\e x_{\a+\xi})|.
\end{align*}

For technical reasons in the proofs it will be useful to have at our disposal a family of ``truncated'' long range energies that we introduce below. For any $M>0$ we set 

\begin{equation}\label{truncenergy}
F^M_{lr,\e}(u,A)=\sum_{|\xi|\leq M}\sum_{\a\in R^{\xi}_{lr,\e}(A)}\e^{n-1}c_{lr}^{\e}(x_{\a},x_{\a+\xi})|u(\e x_{\a})-u(\e x_{\a+\xi})|.
\end{equation}

In what follows we will use the following compactness result for BV-functions taking values in $\{\pm 1\}$ (see \cite{AmBrII} and \cite{AFP} for a general reference)

\begin{theorem}\label{compactemb}
Let $U\subset\rn$ be open and let $u_k\in BV(U,\{\pm 1\})$ be such that

$$\sup_k\mathcal{H}^{n-1}(S(u_k))<+\infty.$$

Then there exists a subsequence $u_k$ (not relabeled) and $u\in BV(U,\{\pm 1\})$ such that $u_k\to u$ strongly in $L^1(U)$.
\end{theorem}

Moreover, the following result on Lipschitz domains turns out to be useful for our proofs (see Theorem 4.1 in \cite{doktor}).

\begin{theorem}\label{lipapprox}
Let $A\in\mathcal{A}^R(\R^n)$. Given $\eta>0$ let $A^{\eta}:=\{x\in \R^{n}:\;\dist(x, A)< \eta\}$. Then, for $\eta$ small enough, $A^{\eta}$ is again a Lipschitz domain and 

\begin{equation*}
\lim_{\eta\to 0}\mathcal{H}^{n-1}(\partial A^{\eta})=\mathcal{H}^{n-1}(\partial A).
\end{equation*} 
\end{theorem}

\begin{remark}\label{alsointerior}
Applying Theorem \ref{lipapprox} to $B_R(0)\backslash A$ with $R$ large enough, we obtain the same result for the set $A_{\eta}:=\{x\in A:\;\dist(x,\partial A)>\eta\}$.
\end{remark}
	
\section{Integral representation}\label{sec:intrep}
We want to make use of the following integral representation theorem which is a special case of Theorem 3 in \cite{BFLM} in the case that the domain of the functional is the space $BV(D,\{\pm 1\})\times\mathcal{A}(D)$ and the functional fulfills a further symmetry property.

\begin{theorem}\label{fonseca}
Let ${\mathcal F}: BV(D,\{\pm 1\})\times {\mathcal A}(D)\to [0,+\infty)$ satisfy for every $(u, A)\in BV(D,\{\pm 1\})\times {\mathcal A}(D)$ the following hypotheses:
\begin{itemize}

\item[(i)] ${\mathcal F}(u, \cdot)$ is the restriction to ${\mathcal A}(D)$ of a Radon measure;

\item[(ii)] ${\mathcal F}(u, A)={\mathcal F}(v, A)$ whenever $u = v$ a.e. on $A\in{\mathcal A}(D)$;

\item[(iii)] ${\mathcal F}(\cdot, A)$ is $\lud$ lower semicontinuous;

\item[(iv)] there exists $C>0$ such that
$$
\frac 1 C \mathcal{H}^{n-1}(S(u)\cap A)\leq  {\mathcal F}(u, A) \leq C \mathcal{H}^{n-1}(S(u)\cap A);
$$

\item[(v)] ${\mathcal F}(u,A)={\mathcal F}(-u,A)$.

\end{itemize}

Then for every $u\in BV(D,\{\pm 1\})$ and $A\in{\mathcal A}(D)$
$$
{\mathcal F}(u, A)=\int_{S(u)\cap A}g(x, \nu_u)\, \mathrm{d}\mathcal{H}^{n-1},
$$
with
$$
g(x_0,\nu)=\limsup_{\rho\to 0}\frac{m(u_{x_0,\nu}, Q_\nu(x_0,\rho))}{\rho^{n-1}},
$$
where
\begin{align*}
u_{x_0,\nu}:=
\begin{cases}
1 & \mbox{if $\langle x-x_0,\nu\rangle\geq 0$,} \\
-1 & \mbox{otherwise,}
\end{cases}
\end{align*}
and for any $(v,A)\in BV(D,\{\pm 1\})\times {\mathcal A}(D)$ we set
\begin{equation*}
m(v,A)= \inf\{ {\mathcal F}(w, A): w\in BV(A;\{\pm 1\}),\ w=v\ \hbox{in a neighborhood of } \partial A\}.
\end{equation*}

\end{theorem}

\hspace*{0,5cm}
The following theorem is the main result of this section. 

\begin{theorem}\label{mainthm1}
Let $\Sigma$ be admissible and let $c_{nn}^{\e}$ and $c_{lr}^{\e}$ satisfy Hypothesis 1. For every sequence $\e_n\to 0^+$ there exists a subsequence $\e_{n_k}$ such that the functionals $F_{\e_{n_k}}$ defined in \eqref{defL1} $\Gamma$-converge with respect to the strong $\lud$-topology to a functional $F:\lud\rightarrow [0,+\infty]$ of the form

\begin{equation*}
F(u)=
\begin{cases}
\int_{S(u)}\phi_{\Sigma}(x,\nu_u)\,\mathrm{d}\mathcal{H}^{n-1} &\mbox{if $u\in BV(D,\{\pm 1\})$,} \\
+\infty &\mbox{otherwise.}
\end{cases}
\end{equation*}

Moreover a local version of the statement above holds: For all $u\in BV(D,\{\pm 1\})$ and all $A\in\Ard$

\begin{equation*}
\Gamma\hbox{-}\lim_k F_{\e_{n_k}}(u,A)=\int_{S(u)\cap A}\phi_{\Sigma}(x,\nu_u)\,\mathrm{d}\mathcal{H}^{n-1}.
\end{equation*}
\end{theorem}

\hspace*{0,5cm}
The proof of Theorem \ref{mainthm1} will be given later. At first we prove several propositions that allow us to apply Theorem \ref{fonseca}. The next two propositions ensure that the limit energy is finite only for $u\in BV(D,\{\pm 1\})$.

\begin{proposition}\label{limitcoercivity}
Let $c_{nn}^{\e}$ and $c_{lr}^{\e}$ satisfy Hypothesis 1. If $A\in\mathcal{A}(D)$ and $u\in\lud$ are such that $F^{\prime}(u,A)<+\infty$, then $u\in BV(A,\{\pm 1\})$  and

$$F^{\prime}(u,A)\geq c\,\mathcal{H}^{n-1}(S(u)\cap A)$$

for some positive constant $c$ independent of $A$ and $u$.
\end{proposition}

\begin{proof}
Let $C_{\e}(\Sigma)\ni u_{\e}\to u$ in $\lud$ be such that $\liminf_{\e\to 0}F_{\e}(u_{\e},A)<+\infty$. Given $\eta>0$, we set $A_{\eta}=\{x\in A:\;\dist(x,\partial A)>\eta\}$. Note that

\begin{equation*}
S(u_{\e})\cap A_{\eta}\subset \bigcup_{\substack{(x,y)\in\NNS \\ \e x,\e y\in (A_{\eta})^{\e R}\\
u_{\e}(\e x)\neq u_{\e}(\e y)}}\e \left(C(x)\cap C(y)\right),
\end{equation*}

so that, by Lemma \ref{cellproperty},

\begin{equation}\label{jumpsetchara}
\mathcal{H}^{n-1}(S(u_{\e})\cap A_{\eta})\leq M_2\sum_{\substack{(\e x,\e y)\in\NNS \\ \e x,\e y\in (A_{\eta})^{\e R}}}\e^{n-1}|u_{\e}(\e x)-u_{\e}(\e y)|.
\end{equation}

By the positivity of $c_{lr}$, Hypothesis 1 and \eqref{jumpsetchara}, for $\e$ small enough we have 

$$
F_{\e}(u_{\e},A) \geq C\sum_{\substack{(x,y)\in\NNS \\ \e x,\e y\in (A_{\eta})^{\e R}}}\e^{n-1}|u_{\e}(\e x)-u_{\e}(\e y)|\geq C\,\mathcal{H}^{n-1}(S(u_{\e})\cap A_{\eta}),
$$

where we have used (\ref{jumpsetchara}). Theorem \ref{compactemb} now implies that $u\in BV_{loc}(A,\{\pm 1\})$ and, since the bound on the measure of the jump set is uniform in $\eta$, we have $u\in BV(A,\{\pm 1\})$. For $\e\to 0$ we get $F^{\prime}(u,A)\geq C\,\mathcal{H}^{n-1}(S(u)\cap A_{\eta})$ by lower semicontinuity. Letting $\eta\to 0$ yields the claim. 

\end{proof}

Before we prove the next upper bound inequality we need to introduce a special class of sets and $BV$-functions.

\begin{definition}\label{polyhedral}
Let $U$ be an open set. A $n$-dimensional polyhedral set in $\rn$ is a set $E\in\mathcal{A}^R(\rn)$ such that its boundary is contained in the union of finitely many affine hyperplanes. A function $u\in BV(U,\{\pm 1\})$ is called a polyhedral function if there exists a $n$-dimensional polyhedral set $E$ in $\rn$ such that $\mathcal{H}^{n-1}(\partial E\cap\partial U)=0$ and $u(x)=1$ if $x\in U\cap E$ and $u(x)=-1$ for every $x\in U\backslash E$.
\end{definition}

\begin{proposition}\label{limitbound}
Let $c_{nn}^{\e}$ and $c_{lr}^{\e}$ satisfy Hypothesis 1. Then there exists a constant $C>0$ such that for all $u\in BV(D,\{\pm 1\})$ and $A\in\Ard$,

$$F^{\prime\prime}(u,A)\leq C\,\mathcal{H}^{n-1}(S(u)\cap A).$$
\end{proposition}

\begin{proof}
We consider only the long-range term. The same technique can be used to bound the nearest neighbour interactions. For the time being we assume that $u_{|A}\in BV(A,\{\pm 1\})$ is a polyhedral function with corresponding set $E$. Note that $S(u)\cap A=\partial E\cap A$. We define $u_{\e}\in C_{\e}(\Sigma)$ by its values on $\e\Sigma$ via 

\begin{equation*}
u_{\e}(\e x):=
\begin{cases}
u(\e x) &\mbox{if $\e x\in A$,}\\
+1 &\mbox{otherwise.}
\end{cases}
\end{equation*}

Then we have $u_{\e}\to u$ in $L^1(A)$. By Hypothesis 1, given $\delta>0$, there exists $M_{\delta}>0$ such that $\sum_{|\xi|>M_{\delta}}J_{lr}(|\hat{\xi}|)(|\xi|+2R)\leq \delta$, where $\hat{\xi}\in \xi+[-r^{\prime},r^{\prime}]^n$ is such that $|\hat{\xi}|=\dist([0,r^{\prime})^n,[0,r^{\prime})^n+\xi)$. Now for all $\xi\in\rzn$ such that $|\xi|\leq M_{\delta}$ and $\eta>0$ there exists $\e_0=\e_0(\delta,\eta)$ such that, for all $\e\leq \e_0$, the couple $(\e x_{\a},\e x_{\a+\xi})$ gives a positive contribution to the energy only if $\dist(\e x_{\a},\partial E\cap A^{\eta})\leq \e(|\xi|+r)$. Using Lemma \ref{cellproperty} we deduce that

\begin{equation*}
\e^{n-1}\#\{\e x_{\a}:\;u(\e x_{\a})\neq u(\e x_{\alpha+\xi}),\,\e x_{\a},\e x_{\a+\xi}\in A\}
\leq \frac{C}{\e}|(\partial E\cap \overline{A^{\eta}})^{\e(|\xi|+2R)}|.
\end{equation*} 
 
Observe that the set $\partial E\cap\overline{A^{\eta}}$ is regular enough to ensure that, at least for small $\e$,

\begin{equation}\label{minkowskicont}
\frac{1}{\e}|(\partial E\cap \overline{A^{\eta}})^{\e(|\xi|+2R)}|\leq C\, \mathcal{H}^{n-1}(\partial E\cap \overline{A^{\eta}})(|\xi|+2R).
\end{equation}

Next we consider the interactions where $|\xi|>M_{\delta}$. For $E$ being a polyhedral set we have

\begin{align}
\#\{\e x_{\a}:\;u(\e x_{\a})\neq u(\e x_{\alpha+\xi}),\,\e x_{\a},\e x_{\a+\xi}\in A\}
\leq \frac{C}{\e^{n-1}}\mathcal{H}^{n-1}(\partial E) (|\xi|+2R).\label{minkowskibound}
\end{align}

Using Remark \ref{localgamma}, (\ref{minkowskicont}), (\ref{minkowskibound}) and the definition of $M_{\delta}$ we conclude that

\begin{align*}
F^{\prime\prime}(u,A)\leq C\,\mathcal{H}^{n-1}(\partial E\cap \overline{A^{\eta}})\sum_{|\xi|\leq M_{\delta}}J_{lr}(|\hat{\xi}|)(|\xi|+2R)+C(E)\, \delta.
\end{align*}

Using the integrability assumption from Hypothesis 1 we infer from the arbitrariness of $\delta$ and $\eta$ that

\begin{equation*}
F^{\prime\prime}(\w)(u,A)\leq C\,\mathcal{H}^{n-1}(S(u)\cap A),
\end{equation*}

where we have used that $\mathcal{H}^{n-1}(\partial E\cap\partial A)=0$.
\\
\hspace*{0,5cm}
For a general function $u\in BV(D,\{\pm 1\})$ let us consider $u_{|A}\in BV(A,\{\pm 1\})$. By the properties of $BV$ functions on Lipschitz domains there exists a sequence of polyhedral functions $u_n\in BV(A,\{\pm 1\})$ such that $u_n\to u_{|A}$ strongly in $L^1(A)$ and $\mathcal{H}^{n-1}(S(u_n)\cap A)\to\mathcal{H}^{n-1}(S(u)\cap A)$. Define $\tilde{u}_n\in\lud$ by

\begin{equation*}
\tilde{u}_n(x)=
\begin{cases}
u_n(x) &\mbox{if $x\in A$,} \\
u(x) &\mbox{otherwise.}
\end{cases}
\end{equation*}

Since $A\in\Ard$ we have that $\tilde{u}_n\in BV(D,\{\pm 1\})$. Moreover $\tilde{u}_n\to u$ strongly in $\lud$ and $\tilde{u}_n$ satisfies the assumptions of the first part of the proof. By the lower semicontinuity of the $\Gamma$-$\limsup$ it holds that

\begin{equation*}
F^{\prime\prime}(u,A)\leq C\liminf_n\mathcal{H}^{n-1}(S(\tilde{u}_n)\cap A)=C\,\mathcal{H}^{n-1}(S(u)\cap A).
\end{equation*}

\end{proof}

In the following proposition we state a weak subadditivity statement for $F''(u,\cdot)$.
\begin{proposition}\label{almostsub}
Let $c_{nn}^{\e}$ and $c_{lr}^{\e}$ satisfy Hypothesis 1. Then, for every $A,B\in\mathcal{A}(D)$, every $A^{\prime}\subset\Ard$ such that $A^{\prime}\subset\subset A$ and every $u\in BV(D,\{\pm 1\})$,

\begin{equation*}
F^{\prime\prime}(u,A^{\prime}\cup B)\leq F^{\prime\prime}(u,A)+F^{\prime\prime}(u,B).
\end{equation*} 
\end{proposition}

\begin{proof}
We only take into account the long-range term, our argument working the same also for short-range interactions. Without loss of generality let $F^{\prime\prime}(u,A)$ and $F^{\prime\prime}(u,B)$ be finite. Let $u_{\e},v_{\e}\in C_{\e}(\Sigma)$ both converge to $u$ in $\lud$ such that

\begin{equation*}
\limsup_{\e\to 0}F_{\e}(u_{\e},A)=F^{\prime\prime}(u,A),\quad\limsup_{\e\to 0}F_{\e}(v_{\e},B)=F^{\prime\prime}(u,B).
\end{equation*}

By Hypothesis 1, given $\delta>0$, there exists $M_{\delta}>0$ such that $\sum_{|\xi|>M_{\delta}}J_{lr}(|\hat{\xi}|)(|\xi|+2R)\leq \delta$. Fix $d\leq\dist(A^{\prime},A^c)$ and let $N_{\e}:=[\frac{d}{\e (M_{\delta}+r)}]$, where $[\cdot]$ denotes the integer part. For $k\in\mathbb{N}$ we define

\begin{equation*}
A_{\e,k}:=\{x\in A:\;\dist(x,A^{\prime})<k\e(M_{\delta}+r)\}
\end{equation*}

and $w^k_{\e}\in C_{\e}(\Sigma)$ by

\begin{equation*}
w^k_{\e}(\e x)=\mathds{1}_{A_{\e,k}}(\e x)u_{\e}(\e x)+(1-\mathds{1}_{A_{\e,k}}(\e x))v_{\e}(\e x).
\end{equation*}

Note that for each fixed $k\in\NN$, $w^k_{\e}\to u$ in $\lud$. Now we set

\begin{equation*}
S_k^{\xi,\e}:=\{x=y+t\,\xi^{\prime}:\;y\in\partial A_{\e,k},\,|t|\leq\e,\xi^{\prime}\in\xi+[-r^{\prime},r^{\prime}]^n\}\cap (A^{\prime}\cup B).
\end{equation*}

For $k\leq N_{\e}$ it can easily be verified that 
 
\begin{align}\label{almostsubeq}
F_{\e}(w^k_{\e},A^{\prime}\cup B)&\leq\, F_{\e}(u_{\e},A_{\e,k})+F_{\e}(v_{\e},A_{\e,k}^c\cap B)
\\
+\sum_{\xi\in\rzn}&\sum_{\a\in R^{\xi}_{lr,\e}(S_k^{\xi,\e})}\underbrace{\e^{n-1}c_{lr}^{\e}(x_{\alpha},x_{\a+\xi})|w^k_{\e}(\e x_{\alpha})-w^k_{\e}(\e x_{\a+\xi})|}_{=:\rho_k^{\xi,\e}(\a)}\nonumber
\\
\leq F_{\e}(\w)(u_{\e},A)+&F_{\e}(\w)(v_{\e},B)
+\sum_{\xi\in\rzn}\sum_{\a\in R^{\xi}_{lr,\e}(S_k^{\xi,\e})}\rho_k^{\xi,\e}(\a).\nonumber
\end{align}

Now if $d$ is small enough (depending only on $A^{\prime}$), by Theorem \ref{lipapprox} we have

\begin{equation*}
\mathcal{H}^{n-1}(\partial A_{\e,k})\leq C\mathcal{H}^{n-1}(\partial A^{\prime})
\end{equation*}

for all $k\leq N_{\e}$. We deduce that

\begin{equation*}
\# (S_k^{\xi,\e}\cap \e\Sigma) \leq C \e^{1-n}\mathcal{H}^{n-1}(\partial A^{\prime})(|\xi|+2R)
\end{equation*}

and hence

\begin{equation}\label{longrange}
\sum_{|\xi|>M_{\delta}}\sum_{\a\in R^{\xi}_{lr,\e}(S_k^{\xi,\e})}\rho_k^{\xi,\e}(\a)\leq C\mathcal{H}^{n-1}(\partial A^{\prime}) \sum_{|\xi|>M_{\delta}}J_{lr}(|\hat{\xi})(|\xi|+2R)\leq C\delta.
\end{equation}

Now we treat the interactions when $|\xi|\leq M_{\delta}$. By our construction we have $S_k^{\e,\xi}\subset (A_{\e,k+1}\backslash A_{\e,k-1})\cap B =: S_k^{\e}$. Observe that every point can only lie in two sets $S_{k_1}^{\e},S_{k_2}^{\e}$. Furthermore a straightforward calculation shows that, for all $x,y\in\e\Sigma$,

\begin{equation*}
|w^k_{\e}(x)-w^k_{\e}(y)|\leq |u_{\e}(x)-u_{\e}(y)|+|v_{\e}(x)-v_{\e}(y)|+2|u_{\e}(y)-v_{\e}(y)|.
\end{equation*}

We deduce that

\begin{align*}
\sum_{|\xi|\leq M_{\delta}}\sum_{\a\in R^{\xi}_{lr,\e}(S_k^{\xi,\e})}\rho_k^{\xi,\e}(\a)
\leq &F_{\e}(u_{\e},S_k^{\e})+F_{\e}(v_{\e},S_k^{\e})
\\
&+C_{\delta}\sum_{\substack{x\in\Sigma \\ \e x\in S_k^{\e}}}\e^{n-1}|u_{\e}(\e x)-v_{\e}(\e x)|,
\end{align*}
where $C_{\delta}$ depends only on $M_{\delta}$. Now averaging the last inequality yields

\begin{align*}
I_{\e}&:=\frac{1}{N_{\e}-1}\sum_{k=1}^{N_{\e}-1}\sum_{|\xi|\leq M_{\delta}}\sum_{\a\in R^{\xi}_{lr,\e}(S_k^{\xi,\e})}\rho_k^{\xi,\e}(\a)
\\
&\leq\frac{2}{N_{\e}-1}\left(F_{\e}(\w)(u_{\e},A)+F_{\e}(\w)(v_{\e},B)\right)+2\,C_{\delta}^{\prime}\sum_{\substack{x\in\Sigma \\
\e x\in D}}\e^{n}|u_{\e}(\e x)-v_{\e}(\e x)|
\\
&\leq\frac{C}{N_{\e}}+C_{\delta}^{\prime}\left(\|u_{\e}-v_{\e}\|_{\lud}+{\scriptstyle \mathcal{O}}(1)\right),
\end{align*}

so that $I_{\e}\to 0$ when $\e\to 0$. For every $\e>0$ let $k_{\e}\in\{1,\dots,N_{\e}-1\}$ be such that

\begin{equation}\label{boundrho}
\sum_{|\xi|\leq M_{\delta}}\sum_{\a\in R^{\xi}_{lr,\e}(S_{k_{\e}}^{\xi,\e})}\rho_{k_{\e}}^{\xi,\e}(\a)\leq I_{\e}
\end{equation}
and set $w_{\e}:=w_{\e}^{k_{\e}}$. Note that $w_{\e}$ still converge to $u$ strongly in $\lud$. Hence, using (\ref{almostsubeq}), (\ref{longrange}) and \eqref{boundrho}, we conclude that

\begin{equation*}
F^{\prime\prime}(u,A^{\prime}\cup B)\leq\limsup_{\e\to 0}F_{\e}(w_{\e},A^{\prime}\cup B)\leq F^{\prime\prime}(u,A)+F^{\prime\prime}(u,B)+C\,\delta.
\end{equation*}

The arbitrariness of $\delta$ proves the claim.
\end{proof}

\begin{proof}[Proof of Theorem \ref{mainthm1}] From Propositions \ref{limitbound} and \ref{almostsub} it follows by standard arguments that $F^{\prime\prime}(u,\cdot)$ is inner regular and subadditive on $\Ard$ (see, for example, Proposition 11.6 in \cite{brde}). Therefore, given a sequence $\e_n\to 0^+$ we can use the compactness property of $\Gamma$-convergence to construct a subsequence $\e_n$ (not relabeled) such that 

\begin{equation*}
\Gamma\hbox{-}\lim_n F_{\e_n}(u,A)=:\tilde{F}(u,A)
\end{equation*}

exists for every $(u,A)\in\lud\times\Ard$. By Proposition \ref{limitcoercivity} we know that $\tilde{F}(u,A)$ is finite only if $u\in BV(A,\{\pm 1\})$. We extend $\tilde{F}(u,\cdot)$ to $\mathcal{A}(D)$ setting

\begin{equation*}
F(u,A):=\sup\,\{\tilde{F}(u,A^{\prime}):\;A^{\prime}\subset\subset A,\, A^{\prime}\in\Ard\}.
\end{equation*}

To complete the proof it is enough to show that $F$ satisfies the assumptions of Theorem \ref{fonseca}. Again by standard arguments $F(u,\cdot)$ fulfills the assumptions of the De Giorgi-Letta criterion so that $F(u,\cdot)$ is the trace of a Borel measure. Since this Borel measure is finite on $D$ by Proposition \ref{limitbound}, it is indeed a Radon measure (Proposition 1.60 in \cite{modernmethods}). The locality property is easy to verify. By the properties of $\Gamma$-limits we know that $\tilde{F}(\cdot,A)$ is $\lud$-lower semicontinuous and so is $F(\cdot,A)$ as the supremum. The growth conditions (iv) in Theorem \ref{fonseca} follow from the Propositions \ref{limitcoercivity} and \ref{limitbound} which still hold for $F$ in place of $\tilde{F}$. Finally the additional symmetry property (v) holds for the discrete energies and thus it is conserved in the limit. The local version of the theorem is a direct consequence of our construction.
\end{proof}

\section{Convergence of boundary value problems}\label{sec:bdy}

In this section we investigate the convergence of the family of functionals $F_{\e}$ under discrete boundary conditions. Let $A\in\Ard$ and let $\varphi\in L^1(\partial A,\{\pm 1\})$. For the sake of simplicity let us assume that $\varphi$ is the trace of a polyhedral function $u_{\varphi}$ (more general boundary conditions may be considered (see Remark \ref{remarkboundary})). In particular we have

\begin{equation}\label{nojump}
\mathcal{H}^{n-1}(S(u_{\varphi})\cap\partial A)=0.
\end{equation}

We define a family of discrete energies with trace constraint as follows: Let $l_{\e}>0$ be such that

\begin{equation}\label{slowboundary}
\lim_{\e\to 0}l_{\e}=+\infty,\quad \lim_{\e\to 0}l_{\e}\e= 0.
\end{equation}

For fixed $u_{\varphi},\,\e>0$ and $l_{\e}>0$ we consider the functional $F^{\varphi,l_{\e}}_{\e}(\cdot,A):\lud\times\Ard\rightarrow[0,+\infty]$ defined as
\begin{align}\label{boundarytype}
F^{\varphi,l_{\e}}_{\e}(u,A):=
\begin{cases}
{F}_{\e}(u,A) &\mbox{if $u(x)=u_{\varphi}(x)$ if $\dist(x,\partial A)\leq l_{\e}\e$.}\\
+\infty &\mbox{otherwise.}
\end{cases}
\end{align}

In order to state and prove the main result of this section we need to introduce additional notation. Let $C_{\e}^{{\varphi},l_{\e}}(\Sigma,A)$ be the space of those $u\in C_{\e}(\Sigma)$ that agree with $u_{\varphi}$ at the boundary of $A$ as:

\begin{equation*}
C_{\e}^{{\varphi},l_{\e}}(\Sigma,A):=\{u\in C_{\e}(\Sigma):\;u(\e x)=u_{\varphi}(\e x)\text{ if }\dist(\e x,\partial A)\leq l_{\e}\e\}.
\end{equation*} 

Furthermore, given $u\in BV(D,\{\pm 1\})$, we set $u_{A,\varphi}:\rn\rightarrow\{\pm 1\}$ as

\begin{equation}\label{tracefunction}
u_{A,\varphi}(x):=
\begin{cases}
u(x) &\mbox{if $x\in A$,}\\
u_{\varphi}(x) &\mbox{otherwise.}
\end{cases}
\end{equation}

Since $A$ is regular we have $u_{A,\varphi}\in BV_{loc}(\rn,\{\pm 1\})$. The following convergence result holds:

\begin{theorem}\label{constrainedproblem}
Let $\Sigma$ be admissible and let $c_{nn}^{\e}$ and $c_{lr}^{\e}$ satisfy Hypothesis 1. For every sequence converging to $0$, let $\e_j$ and $\phi_{\Sigma}$ be as in Theorem \ref{mainthm1}. Assume that the limit integrand $\phi_{\Sigma}$ is continuous on $D\times S^{n-1}$. Then, for every set $A\in\Ard$, $A\subset\subset D$, the functionals $F^{\varphi,l_{\e_j}}_{\e_j}(\cdot,A)$ defined in (\ref{boundarytype}) $\Gamma$-converge with respect to the strong $\lud$-topology to the functional $F^{\varphi}(\cdot,A):\lud\rightarrow [0,+\infty]$ defined by
\begin{equation*}
F^{\varphi}(u,A)=
\begin{cases}
\int_{S(u_{A,\varphi})\cap \overline{A}}\phi_{\Sigma}(x,\nu_{u_{A,\varphi}})\,\mathrm{d}\mathcal{H}^{n-1} &\mbox{if $u\in BV(A,\{\pm 1\})$,}\\
+\infty &\mbox{otherwise.}
\end{cases}
\end{equation*}
\end{theorem}

\begin{proof}
By Proposition \ref{limitcoercivity} we may suppose $u\in BV(A,\{\pm 1\})$.\\

\smallskip

\noindent Proof of the $\liminf$-inequality.\\

\smallskip

We show the lower bound taking only into account the long-range term, the same argument works if we include the short-range interactions at the expenses of heavier notation. Without loss of generality let $u_{j}\to u$ in $\lud$ such that 
\begin{equation*}
\liminf_j F^{\varphi,l_{\e_j}}_{\e_j}(u_{j},A)<C.
\end{equation*}
Hence $u_{j}\in C_{\e_j}^{\varphi,l_{\e_j}}(\Sigma,A)$. Given $\delta>0$, by (\ref{nojump}) and choosing appropriate level sets of the signed distance function of $\partial A$, by Theorem \ref{lipapprox} and Remark \ref{alsointerior} there are $A_1\subset\subset A\subset\subset A_2$ Lipschitz sets such that

\begin{align*}
&\mathcal{H}^{n-1}(S(u_{\varphi})\cap (A_2\backslash \overline{A_1}))\leq\delta,\\
&\mathcal{H}^{n-1}(S(u_{\varphi})\cap\partial A_1)=0,\\
&\mathcal{H}^{n-1}(S(u_{\varphi})\cap\partial A_2)=0.
\end{align*}
Let $u_{\varphi,j}$ be the function defined by

\begin{equation*}
u_{\varphi,j}(\e_j x)=u_{\varphi}(\e_j x).
\end{equation*}

Then $u_{\varphi,j}\to u_{\varphi}$ in $\lud$ and, as in the proof of Proposition \ref{limitbound}, by the choice of $A_1$ and $A_2$ we obtain

\begin{equation}\label{deltaappr}
\limsup_{j}F^{\varphi,l_{\e_j}}_{\e_j}(u_{\varphi,j},A_2\backslash\overline{A_1})\leq C\,\delta.
\end{equation}

We define $\tilde{u}_{j}\in C_{\e_j}(\Sigma)$ by

\begin{equation*}
\tilde{u}_{j}(\e_j x)=\mathds{1}_{A}(\e_j x)u_{j}(\e_j x)+(1-\mathds{1}_{A}(\e_j x))u_{\varphi,j}(\e_j x).
\end{equation*}

Note that $\tilde{u}_j\to u_{A,\varphi}$ in $\lud$. Setting

\begin{equation*}
S^{\xi,j}:=\{x=y+t\,\xi^{\prime}:\;y\in\partial A,\,|t|\leq\e_j,\xi^{\prime}\in\xi+[-r^{\prime},r^{\prime}]^n\}\cap A_2,
\end{equation*}
it holds that 
 
\begin{align}
F_{\e_j}^{\varphi,l_{\e_j}}(\tilde{u}_{\e_j},A_2)&\leq\, F_{\e_j}^{\varphi,l_{\e_j}}(u_{\e_j},A)+F_{\e_j}^{\varphi,l_{\e_j}}(u_{\varphi,\e_j},A_2\backslash \overline{A_1} )\nonumber
\\
+\sum_{\xi\in\rzn}&\sum_{\a\in R^{\xi, A}_{lr,\e_j}(S^{\xi,j})}\e_j^{n-1}c_{lr}^{\e_j}(x_{\alpha},x_{\a+\xi})|\tilde{u}_{j}(\e_j x_{\alpha})-\tilde{u}_{j}(\e_j x_{\a+\xi})|,\label{fundestimate}
\end{align}

where we denoted by 

\begin{equation*}
R^{\xi,A}_{lr,\e_j}(S^{\xi,j})=\left\{\alpha\in R_{lr,\e_{j}}^{\xi}(S^{\xi,j}):\, \e_{j}x_{\alpha}\in A,\, \e_{j}x_{\alpha+\xi}\in D\backslash A\right\}
\end{equation*}
By Hypothesis 1 there exists $M_{\delta}>0$ such that $\sum_{|\xi|>M_{\delta}}J_{lr}(|\hat{\xi}|)(|\xi|+2R)\leq \delta$.  As in the proof of Proposition \ref{almostsub} we obtain
\begin{equation}\label{lrdecay}
\sum_{|\xi|>M_{\delta}}\sum_{R^{\xi,A}_{lr,\e_j}(S^{\xi,j})}\e_j^{n-1}J_{lr}(|\hat{\xi}|)\leq C\mathcal{H}^{n-1}(\partial A)\,\delta.
\end{equation}

For interactions where $|\xi|\leq M_{\delta}$ and $j$ large enough, we have that $S^{\xi,j}\subset A_2\backslash\overline{A_1}$. Moreover, if $l_{\e_j}>M_{\delta}+r$, then by the boundary conditions on $u_j$,

\begin{equation*}
\sum_{|\xi|\leq M_{\delta}}\sum_{\a\in R^{\xi,A}_{lr,\e_j}(S^{\xi,j})}\e_j^{n-1}c_{lr}^{\e_j}(x_{\alpha},x_{\a+\xi})|\tilde{u}_{j}(\e_j x_{\alpha})-\tilde{u}_{j}(\e_j x_{\a+\xi})|
\leq F_{\e_j}^{\varphi,l_{\e_j}}(u_{\varphi,j},A_2\backslash\overline{A_1}).
\end{equation*}

From Theorem \ref{mainthm1} and (\ref{deltaappr}),(\ref{fundestimate}) and (\ref{lrdecay}) we infer

\begin{equation*}
F(u_{A,\varphi},A_2)\leq\liminf_{j}F^{\varphi,l_{\e_j}}_{\e_j}(u_{\e_j},A)+C\,\delta.
\end{equation*}

Now letting $A_2\downarrow \overline{A}$ and then $\delta\to 0$ we obtain the $\liminf$-inequality.
\\
\smallskip

\noindent Proof of the $\limsup$-inequality.\\

\smallskip

We start assuming that $u=u_{\varphi}$ in a neighbourhood of $\partial A$. Let $C_{\e_j}(\Sigma)\ni u_{j}\to u$ in $\lud$ such that

\begin{equation*}
\lim_{j}F_{\e_j}(u_{j},A)=F(u,A).
\end{equation*}
Given $\delta>0$ there exist $M_{\delta}>0$ such that $\sum_{|\xi|>M_{\delta}}J_{lr}(|\hat{\xi}|)(|\xi|+2R)\leq\delta$. Using Theorem \ref{lipapprox} and Remark \ref{alsointerior} we choose regular sets $A_1\subset\subset A_2\subset\subset A$ such that

\begin{align}
&u=u_{\varphi}\quad\text{on }A\backslash \overline{A_1},\label{bordo}\\
&\mathcal{H}^{n-1}(S(u_{\varphi})\cap\partial A_1)=0.\label{polyfunction}
\end{align}

We now proceed by an argument similar to the proof of Proposition \ref{almostsub}. We fix $d\leq\dist(A_1,\partial A_2)$ and set $N_j=[\frac{d}{\e_j(M_{\delta}+r)}]$ and, for $k\in\mathbb{N}$,

\begin{equation*}
A_{j,k}:=\{x\in A:\;\dist(x,A_1)<k\,\e_j\,(M_{\delta}+r)\}.
\end{equation*}

We also define $u^k_j\in C_{\e_j}(\Sigma)$ setting

\begin{equation*}
u^k_j(\e_j x)=
\begin{cases}
u_{\varphi,j}(\e_j x) &\mbox{if $\e_j x\notin A_{j,k}$ and $\dist(\e_j x,A)\leq l_{\e_j}\e_j$,}\\
u_j(\e_j x) &\mbox{otherwise,}
\end{cases}
\end{equation*}

where $u_{\varphi,j}$ is as in the proof of the $\liminf$-inequality. Again we get

\begin{align*}
F_{\e_j}(u^k_j,A)\leq &F_{\e_j}(u_j,A)+F_{\e_j}(u_{\varphi,j},A\backslash \overline{A_1})
\\
&+\sum_{\xi\in\rzn}\e^{n-1}J_{lr}(|\hat{\xi}|)\sum_{\a\in R_{\xi}^{\e_j}(S_k^{\xi,j})}|u^k_j(\e_j x_{\a})-u^k_j(\e_j x_{\a+\xi})|,
\end{align*}

where 
\begin{equation*}
S_k^{\xi,j}:=\{x=y+t\,\xi^{\prime}:\;y\in\partial A_{k,j},\,|t|\leq\e_j,\xi^{\prime}\in\xi+[-r^{\prime},r^{\prime}]^n\}\cap A.
\end{equation*}

As in the proof of Proposition \ref{almostsub} we can show that, at least for small $d$,

\begin{equation*}
\sum_{|\xi|>M_{\delta}}\sum_{\a\in R_{\xi}^{\e_j}(S_k^{\xi,j})}\e^{n-1}J_{lr}(|\hat{\xi}|)\leq C\,\mathcal{H}^{n-1}(\partial A_1)\delta.
\end{equation*}

To control the interactions where $|\xi|\leq M_{\delta}$, we use the averaging technique again to obtain $k_j\in\{1,\dots,N_j\}$ and the corresponding sequence $u^{k_j}_j$ fulfilling the boundary conditions (at least for large $j$ because of (\ref{slowboundary})) and $u^{k_j}_j\to u$ such that

\begin{equation*}
\limsup_j F^{\varphi,l_{\e_{j}}}_{\e_j}(u^{k_j}_j,A)\leq F(u,A)+C\,\mathcal{H}^{n-1}(\partial A_1)\delta+C\,\mathcal{H}^{n-1}(S(u_{\varphi})\cap (A\backslash\overline{A_1})),
\end{equation*}
where we have also used that, by \eqref{bordo}

\begin{equation*}
\limsup_j F^{\varphi,l_{\e_{j}}}_{\e_j}(u_{\varphi,j},A\backslash\overline{A_1})\leq C\,\mathcal{H}^{n-1}(S(u_{\varphi})\cap (A\backslash\overline{A_1})).
\end{equation*}

Letting first $\delta\to 0$ and then $A_1\uparrow A$ we finally get

\begin{equation*}
\Gamma\hbox{-}\limsup_j F^{\varphi,l_{\e_{j}}}_{\e_j}(u,A)\leq F(u,A).
\end{equation*}

\hspace*{0,5cm}
Now given any $u\in BV(A,\{\pm 1\})$ let $u_n$ be the sequence given by Lemma \ref{reshetnyaktype} and let $A^{\prime}\in\Ard$, $A\subset\subset A^{\prime}$. By lower semicontinuity and Reshetnyak's continuity theorem we have

\begin{equation*}
\Gamma\hbox{-}\limsup_j F^{\varphi,l_{\e_{j}}}_{\e_j}(u,A)\leq \liminf_n \left(\Gamma\hbox{-}\limsup_j F^{\varphi,l_{\e_{j}}}_{\e_j}(u_n,A)\right)\leq \liminf_n F(u_n,A^{\prime})=F(u_{A,\varphi},A^{\prime}).
\end{equation*}

Letting $A^{\prime}\downarrow A$ yields the upper bound.
\end{proof}

\begin{theorem}\label{convofminimizers}
Let $A\in\Ard$, $A\subset\subset D$. Under the assumptions of Theorem \ref{constrainedproblem}, the following holds:
\begin{enumerate}
\renewcommand{\labelenumi}{(\roman{enumi})}
\item
\begin{equation*}
\lim_j\left(\inf_{u\in BV(A,\{\pm 1\})}F^{\varphi,l_{\e_j}}_{\varepsilon_j}(u,A)\right)=\min_{u\in BV(A,\{\pm 1\})}F^{\varphi}(u,A).
\end{equation*}
\item Moreover, if $(u_j)_j$ is a converging sequence in $L^1(A,\{\pm 1\})$ such that
\begin{equation*}
F^{\varphi,l_{\e_j}}_{\varepsilon_j}(u_j,A)=\inf_{u\in BV(A,\{\pm 1\})}F^{\varphi,l_{\e_j}}_{\varepsilon_j}(u,A)+{ \scriptstyle \mathcal{O}} (1),
\end{equation*}
then its limit is a minimizer of $F^{\varphi}(\cdot,A)$.
\end{enumerate}
\end{theorem}

\begin{proof}
Note that Theorem \ref{constrainedproblem} also holds when we take the $\Gamma$-limit with respect to the $L^1(A)$-topology. Thus the statement follows immediately from the general theory of $\Gamma$-convergence since the functionals are equicoercive in $L^1(A)$.
\end{proof}

\begin{remark}\label{remarkboundary}\rm{
\item[(i)] If we have only finite range of interactions, i.e. $c_{lr}^{\e}(x,y)=0$ for $|x-y|>L$ then it is enough to take $l_{\e}>L$.
\item[(ii)] Defining an appropriate extension based on an abstract construction using Sard's lemma, it is possible to prove the previous theorem for arbitrary functions $\varphi\in L^1(\partial A,\{\pm 1\})$.
\item[(iii)] If the limit integrand is not continuous in the space variable, one can prove a convergence result for boundary value problems, too. However, in that case one has to define a different discrete trace. In the case of periodic lattices this result is contained in \cite{AlGe14}.
}
\end{remark}

\section{From deterministic to stochastic energies}\label{sec:sto}
So far we have considered energies defined on a fixed, possibly non-periodic network. In this section we replace the fixed lattice $\Sigma$ by a suitable random variable generating the set of points and, as a result of their interaction, random energies. Let $(\Om,\mathcal{F},\mathbb{P})$ be a probability space and let the $\sigma$-algebra $\mathcal{F}$ be complete. \\

In what follows we introduce the stochastic framework that we will use later on to define the random energies.

\begin{definition}
We say that a family $(\tau_z)_{z\in \mathbb{Z}^n},\tau_z:\Om\rightarrow\Om$, is an additive group action on $\Om$ if
\begin{equation*}
\tau_{z_1+z_2}=\tau_{z_2}\circ\tau_{z_1}\quad\forall z_1,z_2\in\mathbb{Z}^n.
\end{equation*} Such an additive group action is called measure preserving if
\begin{equation*}
\mathbb{P}(\tau_z B)=\mathbb{P}(B)\quad \forall B\in\mathcal{F},\,z\in\mathbb{Z}^n.
\end{equation*}
If in addition, for all $B\in\mathcal{F}$ we have
\begin{equation*}
(\tau_z(B)=B\quad\forall z\in \mathbb{Z}^n)\quad\Rightarrow\quad\mathbb{P}(B)\in\{0,1\},
\end{equation*}
then $(\tau_z)_{z\in\mathbb{Z}^n}$ is called ergodic. 
\end{definition}
We set $\mathcal{I}=\{[a,b):a,b\in \mathbb{Z}^{n-1}, a\neq b\}$, where $[a,b):=\{x\in\R^{n-1}:\;a_i\leq x_i<b_i\;\forall i\}$. We introduce the notion of discrete subadditive stochastic processes.

\begin{definition}{\label{discreteprocess}}
A function $\mu:\mathcal{I}\rightarrow L^1(\Omega)$ is said to be a discrete subadditive stochastic process if the following properties hold $\mathbb{P}$-almost surely:

\begin{enumerate}
\renewcommand{\labelenumi}{(\roman{enumi})}
\item \; for every $I\in\mathcal{I}$ and for every finite partition $(I_k)_{k\in K}\subset\mathcal{I}$ of $I$ we have
\begin{equation*}
\mu(I,\omega)\leq \sum_{k\in K}\mu(I_k,\omega).
\end{equation*}
\item \; \begin{equation*}
\inf\left\{\frac{1}{|I|}\int_{\Omega}\mu(I,w)\;\mathrm{d}\mathbb{P}(\omega):\;I\in\mathcal{I}\right\} >-\infty.
\end{equation*}
\end{enumerate}
\end{definition}

Moreover we make use of the following pointwise ergodic theorem (see \cite{ergodic}).

\begin{theorem}{\label{ergodicthm}}
Let $\mu:\mathcal{I}\rightarrow L^1(\Om)$ be a discrete subadditive stochastic process and let $I_k=[-k,k)^{n-1}$. If $\mu$ is stationary with respect to a measure preserving group action $(\tau_z)_{z\in\mathbb{Z}^{n-1}}$, that means
\begin{equation*}
\forall I\in\mathcal{I},\;\forall z\in\mathbb{Z}^{n-1}:\quad\mu(I+z,\omega)=\mu(I,\tau_z\omega)\quad\text{almost surely},
\end{equation*}
then there exists $\Phi:\Omega\rightarrow\mathbb{R}$ such that, for $\mathbb{P}$-almost every $\omega$,
\begin{equation*}
\lim_{k\to +\infty}\frac{\mu(I_k,\omega)}{\mathcal{H}^{n-1}(I_k)}=\Phi(\omega).
\end{equation*}
\end{theorem}

Now we specify the assumptions on the random variable generating the network. 

\begin{definition}
A random variable $\mathcal{L}:\Om\rightarrow (\rn)^{\NN}$, $\w\mapsto \Lw=\{\Lw(i)\}_{i\in\NN}$ is called a stochastic lattice. We say that $\mathcal{L}$ is admissible if $\Lw$ is admissible in the sense of Definition \ref{defadmissible} and the constants $r,R$ can be chosen independent of $\w$ $\mathbb{P}$-almost surely. The stochastic lattice $\mathcal{L}$ is said to be stationary if there exists a measure preserving group action $(\tau_z)_{z\in\mathbb{Z}^n}$ on $\Om$ such that, for $\mathbb{P}$-almost every $\w\in\Om$,
\begin{equation*}
\mathcal{L}(\tau_z\omega)=\Lw+z.
\end{equation*}
If in addition $(\tau_z)_{z\in\mathbb{Z}^n}$ is ergodic, then $\mathcal{L}$ is called ergodic, too.
\end{definition}

As we are interested in proving a stochastic homogenization result, we suppose from now on that there exist functions $c_{nn},c_{lr}:\rn\rightarrow [0,+\infty]$ such that

\begin{equation}\label{periodicityassumpt}
c_{nn}^{\e}(x,y)=c_{nn}(y-x),\quad c_{lr}^{\e}(x,y)=c_{lr}(y-x).
\end{equation}

In particular this means that the coefficients stay deterministic. In the following we let $F_{\e}(\w)$ be the discrete energy defined in the previous section, with the random lattice $\Lw$ in place of $\Sigma$. Many of the notations that will follow in the present random case are the same as those used in the deterministic case in the previous section. As a general rule we will replace $\Sigma$ by $\omega$ to indicate the dependence on the random lattice $\Lw$. For instance $C_{\e}(\w)$ denotes the set defined in \eqref{CepsSig} with $\Lw$ in place of $\Sigma$. \\

The next theorem is the main result of this section.

\begin{theorem}\label{mainthm2}
Let $\mathcal{L}$ be a stationary stochastic lattice and let $c_{nn}$ and $c_{lr}$ satisfy Hypothesis 1 with the additional structure of (\ref{periodicityassumpt}). For $\mathbb{P}$-almost every $\w$ and for all $\nu\in S^{n-1}$ there exists
\begin{equation*}
\phi_{\text{hom}}(\w;\nu):=\lim_{t\to +\infty}\frac{1}{t^{n-1}}\inf\left\{F_1(\w)(u,Q_{\nu}(0,t)):\; u\in C_1^{u_{0,\nu},l_{t^{-1}}}(\w,Q_{\nu}(0,t))\right\}.
\end{equation*}
The functionals $F_{\e}(\w)$ $\Gamma$-converge with respect to the $\lud$-topology to the functional $F_{\text{hom}}(\w):\lud\rightarrow [0,+\infty]$ defined by
\begin{equation*}
F_{\text{hom}}(\w)(u)=
\begin{cases}
\int_{S(u)} \phi_{\text{hom}}(\w;\nu_u)\,\mathrm{d}\mathcal{H}^{n-1} &\mbox{if $u\in BV(D,\{\pm 1\})$,}\\
+\infty &\mbox{otherwise.}
\end{cases}
\end{equation*}
If $\mathcal{L}$ is ergodic, then $\phi_{\text{hom}}(\cdot,\nu)$ is constant almost surely. 
%and is given by
%\begin{equation*}
%\phi_{\text{hom}}(\nu)=\lim_{t\to +\infty}\frac{1}{t^{n-1}}\inf\left\{F_1(\w)(u,Q_{\nu}(0,t)):\; u\in C_1^{u_{0,\nu},l_{t^{-1}}}(\w,Q_{\nu}(0,t))\right\}.
%\end{equation*}
\end{theorem}
\begin{remark}\label{stoc-coef} \rm{The result above holds with the same proof if instead of considering stochastic lattices one takes the deterministic lattice $\mathbb{Z}^n$ and random interaction coefficients $c_{i,j}^{\w}$ for all $i,j\in\mathbb{Z}^n$ such that $\mathbb{P}$-almost surely, it holds that
\begin{align*}
&\frac{1}{C}\leq c_{i,j}^{\w}\leq C \quad\quad\mbox{if $|i-j|=1$,}\\
&0\leq c_{i,j}^{\w}\leq J_{lr}(|i-j|)
\end{align*}
with $C$ and $J_{lr}$ as in Hypothesis 1. (In this setting the assumption on the long range interactions could be weakened a little bit as in \cite{AlGe14}). In this case the stochastic group action acts on the coefficients via
\begin{equation*}
c_{i+z,j+z}^{\w}=c_{i,j}^{\tau_z\w}.
\end{equation*}
In this setting an analogous result has been obtained in \cite{BP} in the case of a two-dimensional system with nearest-neighbors ergodic interactions.}
\end{remark}

\begin{proof} (of Theorem \ref{mainthm2}) We start proving the theorem for the truncated energies 

\begin{equation*}
F_{\e}^L(\w)(u,A):=F_{nn,\e}(\w)(u,A)+F_{lr,\e}^L(\w)(u,A),
\end{equation*}

where $F_{lr,\e}^L$ denotes the truncated long range energy defined in (\ref{truncenergy}). By our assumptions on the lattice, there exists a set $\Om_{\Gamma}^L\subset\Om$ of full measure on which Theorem \ref{mainthm1} holds. Then, for fixed $\w\in\Om_{\Gamma}^L$ and a given sequence $\e_j\to 0$, there exists a subsequence (not relabeled) such that

\begin{equation}\label{possiblelimit}
\Gamma\hbox{-}\lim_{j\to +\infty}F^L_{\e_j}(\w)(u,A)=\int_{S(u)\cap A} \phi^L(\w;x,\nu_u)\,\mathrm{d}\mathcal{H}^{n-1}
\end{equation}

for all $A\in\Ard$ and $u\in BV(A,\{\pm 1\})$. We will prove that, almost surely, the function $\phi^L$ is given by $\phi_{\text{hom}}^L$ for all $x\in D,\nu\in S^{n-1}$, where $\phi_{\text{hom}}^L$ is defined as $\phi_{\text{hom}}$ with the energy $F_{1}$ replaced by $F_{1}^L$ and the width of the discrete boundary reduced to $L+r$. The main part of claim then follows from the Urysohn property of $\Gamma$-convergence.\\

\textbf{Step 1} Existence of $\phi_{\text{hom}}^L$\\

Fix $L\in\mathbb{N}$. We have to show that, for $\mathbb{P}$-almost every $\w\in \Om^L_{\Gamma}$ and every $\nu\in S^{n-1}$, there exists

\begin{equation}\label{trunchom}
\phi_{\text{hom}}^L(\w;\nu)=\lim_{t\to +\infty}\frac{1}{t^{n-1}}\inf\left\{F^L_1(\w)(u,Q_{\nu}(0,t)):\; u\in C_1^{u_{0,\nu},L+r}(\w,Q_{\nu}(0,t))\right\}.
\end{equation}

To reduce notation, for $\nu\in S^{n-1}$ and a cube $Q_{\nu}(x,\rho)$ we set

\begin{equation*}
\mu^L_{\nu}(\w;Q_{\nu}(x,\rho)):=\inf\left\{F^L_1(\w)(u,Q_{\nu}(x,\rho)):\; u\in C_1^{u_{x,\nu},L+r}(\w,Q_{\nu}(x,\rho))\right\}.
\end{equation*}

\textbf{Substep 1.1} Defining a stochastic process\\

At first let $\nu\in S^{n-1}$ be a rational direction. Then there exists a matrix $A_{\nu}\in \mathbb{Q}^{n\times n}$ such that $A_{\nu}e_n=\nu$ and the set $\{A_{\nu}e_j\}_{j=1,\dots,n-1}$ forms an orthonormal basis for $\nu^{\perp}$ (see e.g. \cite{lin}). Moreover there exists an integer $M=M(\nu)>L$ such that $MA_{\nu}(z,0)\in\Zn$ for all $z\in \mathbb{Z}^{n-1}$. To $I=[a_1,b_1)\times \dots\times[a_{n-1},b_{n-1})\in\mathcal{I}$ we associate the set $I_n\subset\rn$ defined by 
\begin{equation*}
I_n:=MA_{\nu}(\text{int}\,I\times (-\frac{s_{\max}}{2},\frac{s_{\max}}{2})),
\end{equation*}
where $s_{\max}=\max_{i}|b_i-a_i|$. We define the stochastic process $\tilde{\mu}^L_{\nu}:\mathcal{I}\rightarrow L^1(\Om)$ setting

\begin{equation}\label{process}
\tilde{\mu}^L_{\nu}(I)(\w):=\inf\left\{F^L_1(\w)(v,I_n):\,v\in C_1^{u_{0,\nu},L+r}(\w,I_n) \right\}+K\,P(I,\mathbb{R}^{n-1}),
\end{equation}
where $P(I,\mathbb{R}^{n-1})$ stands for the perimeter of $I$ in $\mathbb{R}^{n-1}$ and $K$ is a constant to be chosen later.
\\
\hspace*{0,5cm}
Next we show that $\tilde{\mu}^L_{\nu}(I)$ is a $L^1(\Om)$-function. Testing the $C_1(\w)$-interpolation of $u_{0,\nu}$ as candidate in the infimum problem, one can use the growth assumptions from Hypothesis 1 and Lemma \ref{cellproperty} to show that there exists a constant $C>0$ such that

\begin{equation}\label{minestimate}
\tilde{\mu}^L_{\nu}(I)(\w)\leq C\,M^{n-1}\mathcal{H}^{n-1}(I)
\end{equation}

for all $I\in\mathcal{I}$ and every $\w\in\Om_{\Gamma}^L$ so that $\tilde{\mu}^L_{\nu}(I)$ is essentially bounded. The proof of the measurability can be found in the appendix.
\\
\hspace*{0,5cm}
We continue with proving the stationarity of the process. Let $z\in\mathbb{Z}^{n-1}$. Note that $(I-z)_n=I_n-z_M^{\nu}$, where $z_M^{\nu}:=MA_{\nu}(z,0)\in \nu^{\perp}\cap\Zn$. Moreover $v\in C_1^{u_{0,\nu},L+r}(\w,(I-z)_n)$ if and only if $u(\cdot)=v(\cdot-z_m^{\nu})\in C_1^{u_{0,\nu},L+r}(\tau_{z_M^{\nu}}\w,I_n)$. We assume without loss of generality that $r^{\prime}=\frac{1}{k}$ for some positive integer $k$. Note that if the couple $(\a,\xi)$ is taken into account in $\mu^L_{\nu}((I-z)_n)(\w)$ with the corresponding points $x_{\a},x_{\a+\xi}\in\Lw$, then the points $x_{\a^{\prime}}:=x_{\a}+z_M^{\nu}$ and $x_{\a^{\prime}+\xi^{\prime}}:=x_{\a+\xi}+z_M^{\nu}$ are points of the lattice $\mathcal{L}(\tau_{z_M^{\nu}}\w)$ and are taken into account in $\tilde{\mu}^L_{\nu}(I_n)(\tau_{z_M^{\nu}}\w)$ since we have $\a^{\prime}=\a+z_M^{\nu}$ and $\xi^{\prime}=\xi$ so that $|\xi^{\prime}|\leq L$. Here we use the fact that $\mathbb{Z}^n\subset\rzn$. This argument also works the other way around. Moreover $(x_{\a^{\prime}},x_{\a^{\prime}+\xi^{\prime}})\in \NNw$ if and only if $(x_{\a},x_{\a+\xi})\in \mathcal{NN}(\tau_{z_M^{\nu}}\w)$. This shows that $\tilde{\mu}^L_{\nu}(I-z)(\w)=\tilde{\mu}^L_{\nu}(I)(\tau_{z_M^{\nu}}\w)$. Setting $\tilde{\tau}_z=\tau_{-z_M^{\nu}}$ we obtain a measure preserving group action on $\mathbb{Z}^{n-1}$ such that

\begin{equation*}
\tilde{\mu}^L_{\nu}(I)(\tilde{\tau}_z\w)=\tilde{\mu}^L_{\nu}(I+z)(\w).
\end{equation*}

\hspace*{0,5cm}
For the subadditivity let $I\in\mathcal{I}$ and let $\{I^i\}_{i=1}^k\subset\mathcal{I}$ be disjoint such that $I=\bigcup_{i=1}^kI^i$. For fixed $i$ let $u^i\in C_1^{u_{0,\nu},L+r}(\w,I^i_n)$ be such that

\begin{equation*}
\tilde{\mu}^L_{\nu}(I^i)(\w)=F^L_1(\w)(u^i,I^i_n)+K\,P(I_i,\mathbb{R}^{n-1}).
\end{equation*}

We define $u\in C_1^{u_{0,\nu},L+r}(\w,I_n)$ by

\begin{equation*}
u(x):=
\begin{cases}
u^i(x) &\mbox{if $x\in I^i_n$ for some $i$,}\\
u_{0,\nu}(x) &\mbox{otherwise.}
\end{cases}
\end{equation*}

Note that $u$ fulfills the required boundary conditions. Moreover the set $MA_{\nu}(\text{int}\,I\times (-\frac{1}{2},\frac{1}{2}))$ is contained in $\bigcup_{i=1}^kI^i_n$ so that by definition of $u$ and the fact that $M>L+r$ we have

\begin{equation*}
F^L_1(\w)(u,I_n)=F^L_1(\w)(u,\bigcup_{i=1}^kI^i_n).
\end{equation*}

Next if $x_1\in \overline{I^{i_1}_n}$ and $x_2\in\overline{I^{i_2}_n}$ ($i_1\neq i_2)$ give a contribution to $F^L_1(\w)(u,I_n)$, then $|x_1-x_2|\leq L+r$ and therefore $\dist(x_j,\partial I^{i_j}_n)\leq L+r$. We conclude that  $u(x_j)=u_{0,\nu}(x_j)$ for $j=1,2$ so that $x_1$ and $x_2$ lie on different sides of the hyperplane $\nu^{\perp}$. Denoting by $P_{\nu}$ the projection on $\nu^{\perp}$ we know that $|P_{\nu}(x_j)-x_j|\leq L+r$ and $|P_{\nu}(x_1)-P_{\nu}(x_2)|\leq L+r$. Moreover the ray $[P_{\nu}(x_1),P_{\nu}(x_2)]$ intersects a $(n-2)$-dimensional set of the form $\overline{MA_{\nu}I^{k_1}}\cap \overline{MA_{\nu}I^{k_2}}$. It follows that 

\begin{equation*}
\dist(x_j,\bigcup_{1\leq k_1\neq k_2\leq k}\overline{MA_{\nu}I^{k_1}}\cap \overline{MA_{\nu}I^{k_2}})\leq 2(L+r)\quad j=1,2.
\end{equation*} 

We deduce that there exists a constant $C=C(L,M)$ such that 

\begin{equation*}
F^L_1(\w)(u,I_n)\leq \sum_{i=1}^kF^L_1(\w)(u^i,I^i_n)+C\sum_{1\leq i\neq j\leq k}P(\overline{I^{i}}\cap \overline{I^{j}},\mathbb{R}^{n-1}).
\end{equation*}

Having in mind that 

\begin{equation*}
P(I,\mathbb{R}^{n-1})=\sum_{i=1}^kP(I^i,\mathbb{R}^{n-1})-\sum_{1\leq i\neq j\leq k}P(\overline{I^{i}}\cap \overline{I^{j}},\mathbb{R}^{n-1}),
\end{equation*}

we get

\begin{align*}
\tilde{\mu}^L_{\nu}(I)(\w)\leq & F^L_1(\w)(u,I_n)+K\,P(I,\mathbb{R}^{n-1})\\
\leq &\sum_{i=1}^k\tilde{\mu}^L_{\nu}(I^i)(\w)+(C-K)\sum_{1\leq i\neq j\leq k}P(\overline{I^{i}}\cap \overline{I^{j}},\mathbb{R}^{n-1}),
\end{align*}

so that we obtain subadditivity if we choose $K>C$.
The inequality

\begin{equation*}
\inf\left\{\frac{1}{|I|}\int_{\Om}\tilde{\mu}^L_{\nu}(I)(w)\;\mathrm{d}\mathbb{P}(\w):\;I\in\mathcal{I}\right\} >-\infty
\end{equation*}

is trivial since the integrand is always positive. Therefore we can apply Theorem \ref{ergodicthm} to obtain almost surely the existence of $\phi^L_{\text{hom}}(\w;\nu)$ at least for rational directions $\nu$ and the sequence $t_k=2Mk$.\\

\textbf{Substep 1.2} From integer sequences to all sequences\\

Consider an arbitrary sequence $t_k\to +\infty$. From the previous step we know that

\begin{equation*}
\phi^L_{\text{hom}}(\w;\nu)=\lim_{k\to +\infty}\frac{1}{(2M[t_k])^{n-1}}{\mu}^L_{\nu}(\w;2M[t_k]Q_{\nu})
\end{equation*}

almost surely. To reduce notation we set $t_{k,1}=2Mt_k$ and $t_{k,2}=2M[t_k]$. Given a minimizer $u^k\in C_1^{u_{0,\nu},L+r}(\w,t_{k,2}Q)$ for $F^L_1(\w)(\cdot,t_{k,2}Q)$ we set

\begin{equation*}
v^k(x)=
\begin{cases}
u_{0,\nu}(x) &\mbox{if $x\in t_{k,1}Q\backslash t_{k,2}Q$ or $\dist(x,\partial t_{k,1}Q)\leq L+r$,}\\
u^k(x) &\mbox{otherwise.}
\end{cases}
\end{equation*}

We have $v^k\in C_1^{u_{0,\nu},L+r}(\w,t_{k,1}Q)$ and, due to the boundary conditions on $u^k$, also $v^k=u^k$ on $t_{k,2}Q$. Using a similar argument as in the proof of the subadditivity of the process we obtain

\begin{align*}
F^L_1(\w)(v^k,t_{k,1}Q)\leq &F^L_1(\w)(u^k,t_{k,2}Q)+C_L\mathcal{H}^{n-1}((t_{k,1}Q\backslash t_{k,2}Q)\cap\nu^{\perp})\\
&+C_L\mathcal{H}^{n-2}(\partial t_{k,2}Q\cap\nu^{\perp})
\\
\leq &{\mu}^L_{\nu}(\w;t_{k,2}Q_{\nu})+\mathcal{O}(t_{k,1}^{n-2}),
\end{align*}

which yields

\begin{equation}\label{upperbound}
\limsup_{k\to +\infty}\frac{1}{(t_{k,1})^{n-1}}{\mu}^L_{\nu}(\w;t_{k,1}Q_{\nu})\leq\phi^L_{\text{hom}}(\w;\nu).
\end{equation}

Similar we can prove that

\begin{equation}\label{lowerbound}
\phi^L_{\text{hom}}(\w;\nu)\leq\liminf_{k\to +\infty}\frac{1}{(t_{k,1})^{n-1}}{\mu}^L_{\nu}(\w;t_{k,1}Q_{\nu}).
\end{equation}

Combining (\ref{upperbound}) and (\ref{lowerbound}) we get the almost sure existence of the limit for arbitrary sequences.\\

\textbf{Substep 1.3} From rational to irrational directions\\

Let $\Om^L=\bigcap_{\nu\in S^{n-1}\cap \mathbb{Q}^n}\Om_{\nu}^L$ where $\Om_{\nu}^L$ is the set of full measure where the limit exists for the rational direction $\nu$. At first let us prove that the function $\phi^L_{\text{hom}}$ is invariant under the group action $\tau_z$. Given $z\in\mathbb{Z}^n$ and $\w\in\Om^L$ there exists $R=R(L,z)>0$ such that for all $k\in\mathbb{N}$

\begin{equation}\label{movecube}
Q_{\nu}(0,k)\subset Q_{\nu}(-z,R+k),\quad 2(L+r)\leq\dist(\partial Q_{\nu}(0,k),\partial Q_{\nu}(-z,R+k)).
\end{equation}

Similar to the stationarity of the stochastic process we have

\begin{align*}
\phi^L_{\text{hom}}(\tau_z\w;\nu)&\leq\limsup_{k\to +\infty}\frac{1}{(R+k)^{n-1}}\inf\left\{F^L_1(\w)(u,Q_{\nu}(-z,R+k)):\; u\in C_1^{u_{-z,\nu},L+r}(\w,Q_{\nu}(-z,R+k))\right\}
\\
&=\limsup_{k\to +\infty}\frac{1}{k^{n-1}}\inf\left\{F^L_1(\w)(u,Q_{\nu}(-z,R+k)):\; u\in C_1^{u_{-z,\nu},L+r}(\w,Q_{\nu}(-z,R+k))\right\}.
\end{align*}

Now given a minimizer $u_k\in C_1^{u_{0,\nu},L+r}(\w,Q_{\nu}(0,k))$ for $F_1^L(\w)(\cdot,Q_{\nu}(0,k))$ due to (\ref{movecube}) we can extend this function to a function $v_k\in C_1^{u_{-z,\nu},L+r}(\w,Q_{\nu}(-z,R+k))$ such that

\begin{equation*}
F_1^L(\w)(v_k,Q_{\nu}(-z,R+k))\leq F_1^L(\w)(u_k,Q_{\nu}(0,k))+\mathcal{O}(k^{n-2}),
\end{equation*}

hence we get $\phi^L_{\text{hom}}(\tau_z\w;\nu)\leq \phi^L_{\text{hom}}(\w;\nu)$. The other inequality can be proven similar so that the limit indeed exists and, for all $\w\in\Om^L$,

\begin{equation}\label{groupinvariant}
\phi^L_{\text{hom}}(\tau_z\w;\nu)=\phi^L_{\text{hom}}(\w;\nu).
\end{equation}

In particular this shows that $\phi^L_{\text{hom}}(\cdot\,;\nu)$ is measurable with respect to the $\sigma$-algebra $\mathcal{I}$ of invariant sets, that is

\begin{equation*}
\mathcal{J}:=\{A\in\mathcal{F}:\;\mathbb{P}(A\Delta\tau_zA)=0\quad\forall z\in\mathbb{Z}^n\}.
\end{equation*}

\hspace*{0,5cm}
We now show that the definition of $\phi^L_{\text{hom}}$ is independent of the matrix $A_{\nu}$. Indeed, if we consider another cube $\tilde{Q}_{\nu}$ obtained as before but with respect to a different $A_{\nu}$ completing $\nu$ to a different orthonormal basis and the corresponding $\mathcal{J}$-measurable limit $\tilde{\phi}^L_{\text{hom}}$, then it holds

\begin{equation}\label{invariantcharact}
\int_A\phi^L_{\text{hom}}(\w;\nu)\,\mathrm{d}\mathbb{P}(\w)=\int_A\tilde{\phi}^L_{\text{hom}}(\w;\nu)\,\mathrm{d}\mathbb{P}(\w)\quad\quad\forall A\in\mathcal{J},
\end{equation}

so that $\phi^L_{\text{hom}}=\tilde{\phi}^L_{\text{hom}}$ almost surely. Equation (\ref{invariantcharact}) can be justified as follows: Given $k_1>>k_2$ we define 

\begin{align*}
\mathcal{C}_{k_2}&:=\{MA_{\nu}((x,0)+[-\frac{k_2}{2},\frac{k_2}{2})):\;x\in k_2\mathbb{Z}^{n-1}\}
\\
C_{k_1}^{k_2}&:=\bigcup_{\substack{C\in\mathcal{C}_{k_2}\\C\subset \tilde{Q}_{\nu}(0,k_1)}}C
\end{align*}

and the function $v_{k_1}\in C_1^{u_{0,\nu},L+r}(\w,\tilde{Q}_{\nu}(0,k_1))$ 

\begin{equation*}
v_{k_1}(x)=
\begin{cases}
v_C(x) &\mbox{if $x\in C_{k_1}^{k_2}$,}
\\
u_{0,\nu}(x) &\mbox{otherwise,}
\end{cases}
\end{equation*}

where $v_C\in C_1^{u_{0,\nu},L+r}(\w,C)$ is a minimizer for $F_1^L(\w)(\cdot,C)$. Testing $v_{k_1}$ in the definition of the infimum problem defining $\tilde{\phi}^L_{\text{hom}}(\w;\nu)$ we infer that

\begin{align*}
\int_A\tilde{\phi}^L_{\text{hom}}(\w;\nu)\,\mathrm{d}\mathbb{P}(\w)&\leq \liminf_{k_1\to +\infty}\frac{1}{k_1^{n-1}}\int_A F_1^L(\w)(v_{k_1},\tilde{Q}_{\nu}(0,k_1))\,\mathrm{d}\mathbb{P}(\w)
\\
&\leq \limsup_{k_1\to +\infty}\frac{1}{k_1^{n-1}}\int_A C_L^M\left(k_1^{n-2}k_2+\left(\frac{k_1}{k_2}\right)^{n-1}k_2^{n-2}\right)+\sum_{C\subset C_{k_1}^{k_2}}\mu^L_{\nu}(\w;C)\,\mathrm{d}\mathbb{P}(\w)
\\
&\leq \frac{C_L^M}{k_2}+\frac{1}{k_2^{n-1}}\int_A \mu^L_{\nu}(\w;Q_{\nu}(0,k_2))\,\mathrm{d}\mathbb{P}(\w),
\end{align*}

where in the last inequality we have used a change of variables, combined with the stationarity of the process and the fact that $A\in\mathcal{J}$. One inequality in (\ref{invariantcharact}) now follows by letting $k_2\to+\infty$ and applying dominated convergence. The other inequality can be proven the same way.
\\
\hspace*{0,5cm}
Next note that the set of rational directions is dense in $S^{n-1}$. This follows easily from the fact that the inverse of the stereographic projection maps rational points to rational directions. Given $\nu\in S^{n-1}$ and a sequence $t_k\to +\infty$ we define

\begin{align*}
\overline{\phi}_{\text{hom}}^L(\w;\nu)&=\limsup_{k\to +\infty}\frac{1}{t_k^{n-1}}\inf\left\{F^L_1(\w)(u,Q_{\nu}(0,t_k)):\; u\in C_1^{u_{0,\nu},L+r}(\w,Q_{\nu}(0,t_k))\right\},
\\
\underline{\phi}_{\text{hom}}^L(\w;\nu)&=\liminf_{k\to +\infty}\frac{1}{t_k^{n-1}}\inf\left\{F^L_1(\w)(u,Q_{\nu}(0,t_k)):\; u\in C_1^{u_{0,\nu},L+r}(\w,Q_{\nu}(0,t_k))\right\}.
\end{align*}

Let $\nu_j\to\nu$. Since we may construct the cubes $Q_{\nu_j}(0,1)$ such that all basis vectors converge to the basis vectors of $Q_{\nu}(0,1)$ it follows that, for $\w\in\Om^L$, the functions $\overline{\phi}_{\text{hom}}^L,\underline{\phi}_{\text{hom}}^L$ are both continuous extensions of $\phi_{\text{hom}}^L$ on $S^{n-1}$. Indeed, given $\delta>0$ we find $N\in\mathbb{N}$ such that for all $j\geq N$ the following properties hold:

\begin{enumerate}
\item[(i)] $Q_{\nu_j}(0,1-\delta)\subset\subset Q_{\nu}(0,1)\subset\subset Q_{\nu_j}(0,1+\delta)$,
\item[(ii)] $\dist(\nu^{\perp}\cap B_2(0),\nu_j^{\perp}\cap B_2(0))\leq\delta.$
\end{enumerate} 

For fixed $j\geq N$ we define a test function $v_k\in C_1^{u_{0,\nu},L+r}(\w,Q_{\nu}(0,t_k)$ setting

\begin{equation*}
v_k(x):=
\begin{cases}
v_{Q_{\nu_j}(0,(1-\delta)t_k)}(x) &\mbox{if $x\in Q_{\nu_j}(0,(1-\delta)t_k)$,}
\\
u_{0,\nu}(x) &\mbox{otherwise,}
\end{cases}
\end{equation*}

where $v_{Q_{\nu_j}(0,(1-\delta)t_k)}$ is a minimizer for $\mu_{\nu_{j}}^{L}(\w,Q_{\nu_j}(0,(1-\delta)t_k))$. By the choice of $N$ we have

\begin{equation*}
\mu^L_{\nu}(\w,Q_{\nu}(0,t_k)\leq \mu^L_{\nu_j}(\w,Q_{\nu_j}(0,(1-\delta)t_k)+C_L\delta t_k^{n-1}.
\end{equation*}

Dividing the last inequality by $t_k^{n-1}$ and passing to the right subsequence of $t_k$ we deduce

\begin{equation*}
\overline{\phi}_{\text{hom}}^L(\w;\nu)\leq \overline{\phi}_{\text{hom}}^L(\w;\nu_j)(1-\delta)^{n-1}+C_L\delta.
\end{equation*}

Letting first $j\to +\infty$ and then $\delta\to 0$ yields 

\begin{equation*}
\overline{\phi}_{\text{hom}}^L(\w;\nu)\leq \liminf_{j\to +\infty}\overline{\phi}_{\text{hom}}^L(\w;\nu_j).
\end{equation*}

By a symmetric argument we can also prove upper semicontinuity of $\overline{\phi}_{\text{hom}}^L(\w;\cdot)$. The same proof also works for $\underline{\phi}_{\text{hom}}^L(\w;\cdot)$ so we get the existence of the limit for all directions $\nu$ and arbitrary sequences for all $\w\in\Om^L$. Moreover we have proved the continuity of $\nu\mapsto\phi^L_{\text{hom}}(\w;\nu)$.\\

\textbf{Step 2} Translation invariance\\

In this step we prove the existence of the limit defining ${\phi}_{\text{hom}}^L(\w;\nu)$ when we replace $Q_{\nu}(0,t)$ by $Q_{\nu}(x,t)$. We will indeed prove that this limit exists and that it agrees with the one already considered. We start considering a cube $Q_{\nu}(x,1)$ with rational direction $\nu$ and $x\in\mathbb{Z}^n\backslash\{0\}$. Given $\e>0$ we define the events

\begin{equation*}
\mathcal{Q}_N:=\left\{\w\in\Om:\;\forall k\geq \frac{N}{2}\text{ it holds }\left|\frac{1}{k^{n-1}}\mu^L_{\nu}(\w;Q_{\nu}(0,k))-\phi^L_{\text{hom}}(\w;\nu)\right|\leq\e\right\}.
\end{equation*}

By Step 1 we know that the function $\mathds{1}_{\mathcal{Q}_N}$ converges almost surely to $\mathds{1}_{\Om}$. Let us denote by $\mathcal{J}_x$ the $\sigma$-algebra of invariant sets for the measure preserving map $\tau_x$. Fatou's lemma for the conditional expectation yields

\begin{equation}\label{condexp}
\mathds{1}_{\Om}=\mathbb{E}[\mathds{1}_{\Om}|\mathcal{J}_x]\leq\liminf_{N\to +\infty}\mathbb{E}[\mathds{1}_{\mathcal{Q}_N}|\mathcal{J}_x].
\end{equation} 

Using (\ref{condexp}) we know that, given $\delta>0$, we find $N_0=N_0(\w,\delta)$ such that

\begin{equation*}
1\geq\mathbb{E}[\mathds{1}_{\mathcal{Q}_{N_0}}|\mathcal{J}_x](\w)\geq 1-\delta.
\end{equation*} 

Due to Birkhoff's ergodic theorem, almost surely, for every $\gamma>0$ there exists $m_0=m_0(\w,\gamma)$ such that, for any $m\geq \frac{1}{2}m_0$,

\begin{equation*}
\left|\frac{1}{m}\sum_{i=1}^m\mathds{1}_{\mathcal{Q}_{N_0}}(\tau_{i x}\w)-\mathbb{E}[\mathds{1}_{\mathcal{Q}_{N_0}}|\mathcal{J}_x](\w)\right|\leq\gamma.
\end{equation*}

Note that the set we exclude will be a countable union of null sets.
\\
\hspace*{0,5cm}
For fixed $m\geq\max\{m_0(\w,\gamma),N_0(\w,\delta)\}$ we denote by $R$ the maximal integer such that for all $i=m+1,\dots,m+R$ we have $\tau_{ix}(\w)\notin \mathcal{Q}_{N_0}$. In order to bound $R$ let $\tilde{m}$ be the number of unities in the sequence $\{\mathds{1}_{\mathcal{Q}_{N_0}}(\tau_{ix}(\w))\}_{i=1}^{m}$. By definition of $R$ we have

\begin{equation*}
\gamma\geq\left|\frac{\tilde{m}}{m+R}-\mathbb{E}[1_{\mathcal{Q}_{N_0}}|\mathcal{J}_x](\w)\right|=\left|1-\mathbb{E}[1_{\mathcal{Q}_{N_0}}|\mathcal{J}_x](\w)+\frac{\tilde{m}-m-R}{m+R}\right|\geq \frac{R+m-\tilde{m}}{m+R}-\delta.
\end{equation*}

Since $m-\tilde{m}\geq 0$ and without loss of generality $\gamma+\delta\leq\frac{1}{2}$ this provides an upper bound by $R\leq 2m(\gamma+\delta)$. 
\\
\hspace*{0,5cm}
So if we choose an arbitrary $m\geq\max\{m_0(\w,\gamma),N_0(\w,\delta)\}$ and $\tilde{R}=3m(\gamma+\delta)$ we find $l\in [m+1,m+\tilde{R}]$ such that $\tau_{lx}(\w)\in \mathcal{Q}_{N_0}$. Then by (\ref{groupinvariant}) we have for all $k\geq \frac{N_0}{2}$ that

\begin{equation}\label{epsestimate}
\left|\frac{1}{k^{n-1}}\mu^L_{\nu}(\w;Q_{\nu}(-lx,k))-\phi^L_{\text{hom}}(\w;\nu)\right|\leq\e.
\end{equation}

We define $\tilde{l}_1=m+2C_L|x|^2(l-m)$, where $C_L\in\mathbb{N}$ is a constant to be chosen such that $Q_{\nu}(-mx,m)\subset Q_{\nu}(-lx,\tilde{l}_1)$ and $2(L+r)\leq\dist(\partial Q_{\nu}(-mx,m),\partial Q_{\nu}(-lx,\tilde{l}_1))$ (note that $l-m\geq 1$). Then each face of the cube $Q_{\nu}(-mx,m)$ has at most the distance $\tilde{l}_1-m\leq 2C_L|x|^2(l-m)$ to the corresponding face in $Q_{\nu}(-lx,\tilde{l}_1)$. Therefore, given a minimizer $u_m$ on $Q_{\nu}(-mx,m)$ we can again extend $u_m$ to a function $v_m$ on $Q_{\nu}(-lx,\tilde{l}_1)$ such that it fulfills the correct boundary conditions and 

\begin{equation*}
F_1^L(\w)(v_m,Q_{\nu}(-lx,\tilde{l}_1))\leq \mu^L_{\nu}(\w;Q_{\nu}(-mx,m))+C^{\prime}_{L,x}\tilde{R}\,(\tilde{l}_1)^{n-2}.
\end{equation*}

Dividing the last inequality by $(\tilde{l}_1)^{n-1}$ we deduce

\begin{align}
\frac{\mu^L_{\nu}(\w;Q_{\nu}(-lx,\tilde{l}_1))}{(\tilde{l}_1)^{n-1}}&\leq \frac{\mu^L_{\nu}(\w;Q_{\nu}(-mx,m))}{(\tilde{l}_1)^{n-1}}+C^{\prime}_{L,x}\frac{\tilde{R}}{\tilde{l}_1}\nonumber
\\
&\leq \frac{\mu^L_{\nu}(\w;Q_{\nu}(-mx,m))}{m^{n-1}}+3C^{\prime}_{L,x}(\gamma+\delta).\label{firstergodicestimate}
\end{align}

On the other hand we can define $\tilde{l}_2=m-2C_L|x|^2(l-m)$ and use the same argument to obtain

\begin{equation}\label{secondergodicestimate}
\frac{\mu^L_{\nu}(\w;Q_{\nu}(-mx,m))}{m^{n-1}}\leq\frac{\mu^L_{\nu}(\w;Q_{\nu}(-lx,\tilde{l}_2))}{\tilde{l}_2^{n-1}}+3C^{\prime}_{L,x}(\gamma+\delta). 
\end{equation}

Now if $\gamma+\delta$ is small enough (depending only on $x,L$) we have $\tilde{l}_1\geq\tilde{l}_2\geq\frac{m}{2}\geq\frac{N_0}{2}$. Combining (\ref{firstergodicestimate}),(\ref{secondergodicestimate}) and (\ref{epsestimate}) we infer

\begin{equation*}
\limsup_{m\to +\infty}\left|\frac{\mu^L_{\nu}(\w;Q_{\nu}(-mx,m))}{m^{n-1}}-\phi^L_{\text{hom}}(\w;\nu)\right|\leq 3C^{\prime}_{L,x}(\gamma+\delta)+\e,
\end{equation*}

which yields the claim for all integer vectors and all rational directions. Note that the argument also holds if we consider a cube $Q_{\nu}(x,\rho)$ with $\rho\in\mathbb{Q}$ where the constants may also depend on $\rho$ (taking out another null set).\\ 

The extension of this result to the general case is straightforward. Here we make only few comments and leave the details to the reader. 
The extension to arbitrary sequences (and thus to rational points) works similar as for the case of cubes centered at the origin except that now the cubes to be compared are not contained in each other (this is a minor detail that can be fixed by the same arguments already used in the proof of the invariance under the group action). For irrational points and irrational directions $\nu\in S^{n-1}$ the statement follows by continuity arguments.\\

\textbf{Step 3 }$\Gamma$-convergence for the truncated energies\\

So far we have shown that for every $L\in\mathbb{N}$ there exists a set $\Om^L$ of full measure such that for all $\nu\in S^{n-1}$ there exists

\begin{equation}\label{theclue2}
\phi^L_{\text{hom}}(\w;\nu)=\lim_{t\to +\infty}\frac{1}{t^{n-1}}\inf\left\{F_1^L(\w)(v,Q_{\nu}(tx,t)):\;v\in C_1^{u_{tx,\nu},L+r}(\w,Q_{\nu}(tx,t))\right\}
\end{equation}

for all $x\in D$ and all $\w\in\Om^L$. As we show now, this function is the surface density of every possible $\Gamma$-limit of the truncated energies. \\

We start with the lower bound. Fix $\w\in \Om^L$. Using Theorem \ref{fonseca}, for every $\eta>0,x\in D$ and $\nu\in S^{n-1}$ we find $\rho=\rho_{\eta}>0$ and a function $u_{\eta}\in BV(Q_{\nu}(x,\rho),\{\pm 1\})$ such that $u_{\eta}=u_{x,\nu}$ in a neighbourhood of $\partial Q_{\nu}(x,\rho)$ and

\begin{equation*}
\phi^L(\w;x,\nu)+\eta\geq\frac{1}{\rho^{n-1}}\int_{S(u_{\eta})\cap Q_{\nu}(x,\rho)}\phi^L(\w;y,\nu_{u_{\eta}})\,\mathrm{d}\mathcal{H}^{n-1}.
\end{equation*}

Now from the construction of the recovery sequence in the proof of Theorem \ref{constrainedproblem} we know that we can take a recovery sequence for $u_{\eta}$ already fulfilling the discrete boundary conditions. Let $u_{\eta,j}$ be such a recovery sequence. From the last inequality we infer that

\begin{align*}
\phi^L(\w;x,\nu)+\eta&\geq\frac{1}{\rho^{n-1}}\lim_{j\to +\infty}F_{\e_j}(\w)(u_{\eta,j},Q_{\nu}(x,\rho))
\\
&\geq \liminf_{j\to +\infty}\frac{1}{\rho^{n-1}}\inf\{F_{\e_j}(\w)(v,Q_{\nu}(x,\rho)):\;v\in C_{\e_j}^{u_{x,\nu},L+r}(\w,Q_{\nu}(x,\rho))\}=\phi^L_{\text{hom}}(\w;\nu),
\end{align*}

where the last equality follows by a simple rescaling of the energies. Since $\eta>0$ was arbitrary, we proved the lower bound.\\

For the upper bound, note that by the continuity of $\nu\mapsto\phi^L_{\text{hom}}(\w;\nu)$ it is enough to prove it for polyhedral functions. Moreover, since the construction we provide is local we can assume that $u\in BV(D,\{\pm 1\})$ with $S(u)\cap D=(x+\nu^{\perp})\cap D$ for some $x\in D$ and $\nu\in S^{n-1}$. Again let $\eta>0$. We cover $S(u)\cap D$ with cubes $\{Q_{\nu}(x_i,\eta)\}_{i=1}^{M_{\eta}}$, where $x_i\in S(u)\cap D$ and $M_{\eta}=\eta^{1-n}\left(\mathcal{H}^{n-1}(S(u)\cap \overline{D})+{ \scriptstyle \mathcal{O}} (1)\right)$. Now we define the sequence $u_{\eta,j}\in C_{\e_j}(\w)$ as

\begin{equation*}
u_{\eta,j}=
\begin{cases}
u(\e_j x) &\mbox{if $x\notin \bigcup Q_{\nu}(x_i,\eta)$,}\\
u_j^i(\e_j x) &\mbox{if $x\in Q_{\nu}(x_i,\eta)$,}
\end{cases}
\end{equation*}

where $u_j^i$ is the rescaled solution of the problem

\begin{equation*}
\inf\left\{F_1^L(\w)(v,Q_{\nu}(\frac{x}{\e_j},\frac{\eta}{\e_j})):\;v\in C_1^{u_{\frac{x}{\e_j},\nu},L+r}(\w,Q_{\nu}(\frac{x}{\e_j},\frac{\rho}{\e_j}))\right\},
\end{equation*}

and $\nu$ is orientated in the obvious way. By the previous steps we have that the contribution on each cube $Q_{\nu}(x_i,\eta)$ converges to $\eta^{n-1}\phi^L_{\text{hom}}(\w;\nu)$. It remains to control the interactions between different cubes. Since on each cube the functions fulfill boundary conditions on a discrete boundary whose width is larger than the range of interactions, interacting points in different cubes are bound to be in a set whose measure scales like $\e^n$ and therefore their contribution to the energy vanishes in the limit. Moreover we have no contributions from outside the cubes. This follows again by the width of the discrete boundary. By (\ref{minestimate}) it follows that the sequence has bounded energy and arguing similar to the proof of Proposition \ref{limitcoercivity} there exists a subsequence $\e_{j_k}$ such that $u_{\eta,j_k}$ converges to some $u_{\eta}$. Taking a countable set $\eta_l\to 0$, by a diagonal argument this holds for all $\eta_l$. We conclude that

\begin{align*}
\Gamma-\lim_j F_{\e_j}(\w)(u_{\eta_l})=\Gamma-\lim_k F_{\e_{j_k}}(\w)(u_{\eta_l})\leq \limsup_{k}F_{\e_{j_k}}(\w)(u_{\eta_l,j_k})\leq M_{\eta_l}\eta_l^{n-1}\phi^L_{\text{hom}}(\w;\nu).
\end{align*}

Now obviously $u_{\eta}$ converges to $u$ in $L^1(D)$ when $\eta\to 0$. By lower semicontinuity we deduce that

\begin{equation*}
\Gamma-\lim_j F_{\e_j}(\w)(u)\leq \phi^L_{\text{hom}}(\w;\nu)\mathcal{H}^{n-1}(S(u)\cap\overline{D}),
\end{equation*}

which yields the result since $u$ is a polyhedral function and therefore $\mathcal{H}^{n-1}(S(u)\cap \partial D)=0$.\\

\textbf{Step 4} Eliminating the truncation\\

We have shown that the truncated energies possess a $\Gamma$-limit with a surface density independent of $x$. Therefore Theorems \ref{constrainedproblem} and \ref{convofminimizers} hold. Using Remark \ref{remarkboundary} we can replace the width of the boundary by some sequence $l_{\e}$ fulfilling (\ref{slowboundary}) in the definition of $\phi^L_{\text{hom}}(w;\nu)$. Having this in mind we now show that we can define $\phi_{\text{hom}}(\w;\nu)$ as the pointwise limit of the functions $\phi^L_{\text{hom}}(\w;\nu)$ on the set $\Om_{\Gamma}=\bigcap_L \Om^L$.
\\
\hspace*{0,5cm}
Indeed let $t_k\to+\infty,\,\w\in\Om^{\Gamma}$ and let $Q=Q_{\nu}(x,\rho)$ be a cube. We set $l_k=l_{t_k^{-1}}$ and
\begin{align*}
&\mu_k(\w):=\inf\left\{F_1(\w)(v,t_kQ):\;v\in C_1^{u_{t_kx,\nu},l_k}(\w,t_kQ)\right\},\\
&\mu_k^L(\w):=\inf\left\{F_1^L(\w)(v,t_kQ):\;v\in C_1^{u_{t_kx,\nu},l_k}(\w,t_kQ)\right\}.
\end{align*}

Let $v^L_k\in C_1^{u_{t_kx,\nu},l_k}(\w,t_kQ)$ be such that

\begin{equation*}
F_1^L(\w)(v^L_k,t_kQ)=\mu_k^L(\w).
\end{equation*}

We have
\begin{align*}
0&\leq\frac{\mu_k(\w)-\mu_k^L(\w)}{t_k^{n-1}}\leq\frac{F_1(\w)(v^L_k,t_kQ)-F_1^L(\w)(v^L_k,t_kQ)}{t_k^{n-1}}\\
&\leq\frac{1}{t_k^{n-1}}\left(\sum_{|\xi|>L}J_{lr}(|\hat{\xi}|)\sum_{\alpha\in R^{\xi}_{lr,1}(t_kQ)}|v^L_k(x_{\alpha})-v^L_k(x_{\alpha+\xi})|\right).
\end{align*}

We can bound the last expression by

\begin{equation*}
\frac{C\,\mathcal{H}^{n-1}(S(v^L_k)\cap t_kQ)}{t_k^{n-1}}\sum_{|\xi|>L}J_{lr}(|\hat{\xi}|)(|\xi|+2R).
\end{equation*}

Using the fixed values of $v^L_k$ near the boundary (we may assume that $l_k>2R$) and the same argument as in the proof of Proposition \ref{limitcoercivity}, one can show that 

\begin{equation*}
\mathcal{H}^{n-1}(S(v^L_k)\cap t_kQ)\leq C\,F_1^L(\w)(v^L_k,t_kQ)+C\,t_k^{n-1},
\end{equation*}

so that, taking into account (\ref{minestimate}) we infer that

\begin{equation}\label{truncestimate}
0\leq\frac{\mu_k(\w)-\mu_k^L(\w)}{t_k^{n-1}}\leq C\sum_{|\xi|>L}J_{lr}(|\hat{\xi}|)(|\xi|+2R).
\end{equation}

This shows that $\{\phi^L_{\text{hom}}(\w;\nu)\}_L$ is a Cauchy sequence because $\frac{\mu_k^L(\w)}{t_k^{n-1}}$ converges to $\frac{\mu_k(\w)}{t_k^{n-1}}$ uniformly in $k$. Defining $\phi_{\text{hom}}(\w;\nu):=\lim_{L\to +\infty}\phi^L_{\text{hom}}(\w;\nu)$ it follows easily from (\ref{truncestimate}) that 

\begin{equation*}
\phi_{\text{hom}}(\w;\nu)=\lim_{k\to +\infty}\frac{\mu_k(\w)}{t_k^{n-1}}.
\end{equation*}

Since $\phi_{\text{hom}}(\w;\cdot)$ is the limit of surface integrands of $\Gamma$-limits, it is convex (in the sense of one-homogeneous extensions) and bounded, hence continuous and we may proceed as in Step 3 to show that the surface density of every possible $\Gamma$-limit is given by $\phi_{\text{hom}}(\w;\nu)$. The proof of the lower bound remains unchanged while for the recovery sequence we split the contributions between two different cubes to finite range interactions within distance smaller than $M$ and the long range interactions that can be bound by $C\sum_{|\xi|\geq M}J_{lr}(|\hat{\xi}|)(|\xi|+2R)\mathcal{H}^{n-1}(\partial Q_{\nu}(x_i,\eta))$. Since the width of the boundary fulfills (\ref{slowboundary}) we can apply the argument from the previous step for fixed, but arbitrarily large $M$. We let first $j\to +\infty$ and then $M\to +\infty$. From then on we proceed as in the case of finite range interactions.\\

The claim of the theorem regarding the ergodic case follows easily from (\ref{groupinvariant}) since in this case all the functions $\phi^L_{\text{hom}}(\cdot\,;\nu)$ are constant and so is the pointwise limit.\end{proof}

\section{Examples and generalizations}\label{sect:last}
In this section, motivated by the applications, we provide two examples of stochastic lattices that are admissible in our setting. In the last paragraph we generalize the results obtained so far for pairwise interaction models to multi-body interactions.

\subsection{Stochastically mixing perturbations of periodic lattices with defects}

We may model a magnetic crystal as a discrete spin system parameterized on a periodic multi-lattice\ie the union of finitely many translated copies of $\Zn$. These lattices include for example the well-known fcc, hcp or bcc lattices in dimension $n=3$ or the triangular lattice in dimension $n=2$. We allow these crystalline structures to have defects in the periodic configuration and we assume that the number of defects is small with respect to our relevant scaling, that is surface scaling. For small temperature we expect small random fluctuations of the position of the atoms in the network to take place. The way we model such fluctuations as small stochastic perturbations of the crystalline lattice is detailed in the following construction.\\

Given a fixed periodic lattice, we need to construct probability space that would generate the stochastic lattice. We make this construction explicitly in the case of $\Zn$. Let $\mu$ be a probability measure on $[-a,a]^n$, where $|a|<\frac{1}{2}$. We define $\Om=\Pi_{i\in\Zn}[-a,a]^n$ and let $(\mathcal{F},\mathbb{P})$ be the product $\sigma$-algebra with the product measure. As a group action we take the shift operator. This action is measure preserving and strongly mixing on cylindrical sets and therefore on all sets, whence ergodic. We define the stochastic lattice as
\begin{equation*}
\mathcal{L}(\w)_i=i+\w_i.
\end{equation*}
The bound on $|a|$ ensures that $\Lw$ is admissible uniformly in $\w$ almost surely. Moreover, by the construction we have $\Lw+z=\mathcal{L}(\tau_z\w)$ at least as a point set, so that Theorem \ref{mainthm2} can be applied. As the construction done in the case of $\Zn$ can be repeated for any periodic multi-lattice $\mathcal{L}_{p}$ (with some obvious modification on the bound on $|a|$), in what follows we state the results for any multi-lattice.\\

We now model the possible presence of defects on the periodic structure considering an admissible lattice $\mathcal{L}$ (the lattice with defects) with the property that there exists a periodic multi-lattice $\mathcal{L}_p$ (a lattice without defects) such that

\begin{equation}\label{almostperiodic}
\lim_{N\to +\infty}\frac{\#\left(\mathcal{L}\Delta\mathcal{L}_p\right)\cap B_N(0)}{N^{n-1}}=0.
\end{equation}
Note that here the number of defects is scaled to zero with respect to the surface scaling, which is the relevant scaling of our problem. \\

In the next proposition we show that the stochastic perturbation of the lattice $\mathcal L$ yields the same asymptotic energy as that of the stochastically perturbed multi-lattice $\mathcal{L}_{p}$.
\begin{proposition}\label{changelattice}
Let $\Lw$ fulfill (\ref{almostperiodic}) with corresponding lattice $\Lw_p$. Then, for every $u\in BV(D,\{\pm 1\})$ and every cube $Q$, we have 
\begin{equation*}
F(\w)(u,Q)=F(\w)_p(u,Q),
\end{equation*}
where $F(\w)_p$ denotes the limit energy with respect to the lattice $\Lw_p$ (we assume it exists). 
\end{proposition}

\begin{proof}
We only show that $F(\w)(u,Q)\leq F(\w)_p(u,Q)$, the other inequality being exactly the same. Using the inner regularity of the limit functional, it is enough to show that $F(\w)(u,Q^{\prime})\leq F(\w)_p(u,Q)$ for every cube $Q^{\prime}\subset\subset Q$.
\\
\hspace*{0,5cm}
Given $\delta>0$ there exists $M_{\delta}>0$ such that $\sum_{|\xi|>M_{\delta}}J_{lr}(|\hat{\xi}|)(|\xi|+2R)\leq \delta$. Next we have to check how many interacting points change from nearest neighbours to long-range or vice versa, or vanish when changing the lattice. Note that $(x_{\a},x_{\a+\xi})\in \NNw$ if and only if

\begin{equation}\label{charann}
\exists z\in\rn :\;|x_{\a}-z|=|x_{\a+\xi}-z|<|l-z|\quad\forall l\in\Lw\backslash\{x_{\a},x_{\a+\xi}\}.
\end{equation} 
We have to distinguish several cases: If $(x_{\a},x_{\a+\xi})\in \NNw$ then there are three possibilities:

\begin{itemize}
\item [(i)] $(x_{\a},x_{\a+\xi})\in \NNw_p$,
\item [(ii)] $x_{\a}\notin\Lw_p$ or $x_{\a+\xi}\notin\Lw_p$,
\item [(iii)] $(x_{\a},x_{\a+\xi})\notin \NNw_p$, $\{x_{\a},x_{\a+\xi}\}\subset \Lw_p$,
\end{itemize}
where $\NNw_p$ means the nearest neighbours in the lattice $\Lw_p$. In the last two cases one can deduce that

\begin{equation}\label{newlattice1}
\{x_{\a},x_{\a+\xi}\}\subset\left(\Lw\Delta\Lw_p\right)^{2R}.
\end{equation}
In the second case this is trivial, in the third case one uses (\ref{charann}) to see that there exists $z\in\rn$ and $l_p\in\Lw_p\backslash \Lw$ such that
\begin{align*}
|x_{\a}-z|&=|x_{\a+\xi}-z|<|l-z|\quad\forall l\in\Lw\backslash\{x_{\a},x_{\a+\xi}\},\\
|x_{\a}-z|&=|x_{\a+\xi}-z|\geq |l_p-z|,
\end{align*}
from which we deduce that
\begin{align*}
|l_p-x_{\a}|&\leq |l_p-z|+|x_{\a}-z|\leq 2|x_{\a}-z|\leq 2R,\\
|l_p-x_{\a+\xi}|&\leq |l_p-z|+|x_{\a+\xi}-z|\leq 2|x_{\a+\xi}-z|\leq 2R.
\end{align*}
where we have used that $z\in C(x)\cap C(y)$.
\\
\hspace*{0,5cm}
If $(x_{\a},x_{\a+\xi})\notin NN(\w)$ we have again the three possibilities from above. In the two interesting cases one can use a similar argument as before to show that (for $M_{\delta}$ large enough, depending on $\Lw_p$)
\begin{equation}\label{newlattice2}
\{x_{\a},x_{\a+\xi}\}\subset\left(\Lw\Delta\Lw_p\right)^{M_{\delta}+2R} \text{     or     }|\xi|>M_{\delta}.
\end{equation}

Now let $u_{\e,p}$ converge to $u$ in $\lud$ such that

\begin{equation*}
\lim_{\e\to 0}F_{\e}(\w)_p(u_{\e},Q)=F(\w)_p(u,Q).
\end{equation*}
We define $u_{\e}\in C_{\e}(\w)$ by

\begin{equation*}
u_{\e}(\e x)=
\begin{cases}
u_{\e,p}(\e x) &\mbox{if $\e x\in\e\Lw_p$,}\\
+1 &\mbox{otherwise.}
\end{cases}
\end{equation*}

It follows easily that also $u_{\e}\to u$ in $\lud$. Using (\ref{newlattice1}), (\ref{newlattice2}) and Hypothesis 1 we get

\begin{align*}
F_{\e}(\w)(u_{\e},Q^{\prime})\leq& F_{\e}(\w)_p(u_{\e,p},Q)+C\,\e^{n-1}\#\{(x,y)\in \left(\Lw\Delta\Lw_p\right)^{M_{\delta}+2R}\cap \frac{1}{\e}Q\}
\\
&+\sum_{|\xi|>M_{\delta}}J_{lr}(|\hat{\xi}|)\sum_{\a\in R_{lr,\e}^{\xi}(Q^{\prime})}\e^{n-1}|u_{\e,p}(\e x_{\a})-u_{\e,p}(\e x_{\a+\xi})|.
\end{align*}
Using Lemma \ref{cellproperty} and (\ref{almostperiodic}) one can easily see that the second term vanishes when $\e\to 0$. To estimate the last term note that, since $Q^{\prime}$ is convex, $u_{\e,p}(\e x_{\a})\neq u_{\e,p}(\e x_{\a+\xi})$ implies that $[\e x_{\a},\e x_{\a+\xi}]\cap S(u_{\e,p})\cap Q^{\prime}\neq\emptyset$. This yields

\begin{equation}\label{nearjumpset}
\sum_{\a\in R_{lr,\e}^{\xi}(Q^{\prime})}\e^{n-1}|u_{\e,p}(\e x_{\a})-u_{\e,p}(\e x_{\a+\xi})|\leq C\,\mathcal{H}^{n-1}(S(u_{\e,p})\cap Q^{\prime})(|\xi|+2R),
\end{equation} 
where we have used that $S(u_{\e,p})$ is regular. Repeating the proof of Proposition \ref{limitcoercivity} we know that $\mathcal{H}^{n-1}(S(u_{\e,p})\cap Q^{\prime})$ is bounded. By the definition of $M_{\delta}$ we have shown that 

\begin{equation*}
F(\w)(u,Q^{\prime})\leq F(\w)_p(u,Q)+C\,\delta.
\end{equation*}
The arbitrariness of $\delta$ proves the claim.
\end{proof}

\subsection{Random parking and isotropy}
It is a very challenging problem to relate the symmetries of the stochastic lattice to those of the limit energy density. It has already been observed in the periodic setting in \cite {ABC} that for Ising spin systems the discrete symmetries of the periodic lattice induces anisotropies in the limit. In the stochastic setting, on the other hand, one can imagine that the anisotropy of a single realization of the stochastic lattice may be averaged out by ergodicity. This is indeed the case of another interesting and more involved probabilistic setup: the so called random parking model investigated in \cite{glpe}. This model provides a stochastic lattice that is stationary, ergodic and in addition stationary with respect to rotations in the following sense: for all $R\in SO(n)$ there exists a measure preserving group action $\tau_R:\Om\rightarrow\Om$ such that

\begin{equation}\label{stochasticisotropy}
\mathcal{L}(\tau_R\w)=R\Lw.
\end{equation}

If the discrete energy densities are isotropic in the space variable, that means

\begin{equation}\label{spaceisotropy}
c_{nn}(z)=c_{nn}(|z|),\quad c_{lr}(z)=c_{lr}(|z|)\quad\forall z\in\R^n, 
\end{equation}

then one would expect that the limit energy is isotropic, too. Indeed, the following theorem holds.

\begin{theorem}\label{isotropicenergy}
Let $\mathcal{L}$ be a stationary (with respect to both translations and rotations) stochastic lattice and let $c_{nn},c_{lr}$ satisfy Hypothesis 1 and (\ref{spaceisotropy}). Then, for $\mathbb{P}$-almost every $\w\in\Om$,  $\phi_{\text{hom}}(\w;\cdot)$ is constant and the $\Gamma$-limit of the functionals $F_{\e}(\w)$ is given by

\begin{equation*}
F(\w)(u)=
\begin{cases}
\phi_{\text{hom}}(\w)\mathcal{H}^{n-1}(S(u)) &\mbox{if $u\in BV(D,\{\pm 1\})$,}\\
+\infty &\mbox{otherwise.}
\end{cases}
\end{equation*}
\end{theorem}

\begin{proof}
Fix $\nu\in S^{n-1}$ and consider countably many rotations $\{R_n\}$ such that $\{R_n\nu\}$ is dense in $S^{n-1}$. By continuity it suffices to prove that there exists a set $\Om^{\prime}$ of full measure such that $\phi_{\text{hom}}(\w;R_n\nu)=\phi_{\text{hom}}(\w;\nu)$ for all $n\in\mathbb{N},\w\in\Om^{\prime}$. Let $\Om^{\prime\prime}$ be the set of full measure where $\phi_{\text{hom}}(\w;\nu)$ exists. Since $\tau_{R_n}$ is measure preserving, it follows there exists a set $\Om^{\prime}_n\subset\Omega^{\prime\prime}$ such that also $\phi_{\text{hom}}(\tau_{R_n}\w;\nu)$ exists for all $\w\in\Omega^{\prime}_n$. If we define $\Omega^{\prime}=\bigcap_n\Omega^{\prime}_n$, then the claim follows by the same arguments as in the proof of Theorem 9 in \cite{ACG2}.
\end{proof}

\begin{remark}\rm{
\item[(1)]Under the additional ergodicity assumption Theorem \ref{isotropicenergy} shows that the $\Gamma$-limit turns out to be of the form $C\,\mathcal{H}^{n-1}(S(u))$ where $C$ is a deterministic constant. 
\item[(2)] Combining the previous remark with the results on non-local energies obtained in \cite{AlGe14}, the lattice given by the random parking model can be used to obtain an approximation of the Ohta-Kawasaki energy. 
}
\end{remark}

\subsection{Multi-body interactions}

Our techniques can also be applied to treat the variational convergence of energies accounting for multi-body interactions. To be precise, given $M\in\mathbb{N}$, an admissible lattice $\Sigma=\Sigma(r,R)$ and a function $u\in C_{\e}(\Sigma)$ we can consider an (already localized) energy of the form

\begin{equation}\label{multibody}
F_{\e,M}(u,A)=\sum_{\substack{x=(x_1,\dots,x_m)\in\Sigma^M\\ \e x_1,\dots,\e x_M\in A}}\e^{n-1}f^{\e}(x,u(\e x_1),\dots,u(\e x_M)),
\end{equation}

Let us set $\underline{1}=(1,\dots,1)\in\mathbb{R}^M$. Analogous to the case of pairwise interactions, we make the following assumptions on the density $f^{\e}:(\rn)^M\times\{\pm 1\}^M\rightarrow [0,+\infty)$: There exists a monotone decreasing function $J:[0,+\infty)\rightarrow [0,+\infty)$ with

\begin{equation*}
\int_{(\rn)^{M-1}}J(|x|_{\infty})|x|_{\infty}\,\mathrm{d}x <+\infty
\end{equation*}

such that
\\
\begin{itemize}
\item[(M1)] $f^{\e}(x_1,\dots,x_M,\pm\underline{1})=0\quad\forall x_1,\dots,x_M\in\rn,\quad$   (ferromagnetic behaviour)
\\
\item[(M2)] $\min\{f^{\e}(x_1,\dots,x_M,u_1,\dots,u_M):\;\sup_i|x_1-x_i|\leq 2R,\,\exists i:\;u_i\neq u_1\}\geq c>0,\quad$   (coercivity)
\\
\item[(M3)] $f^{\e}(x_1,\dots,x_M,u_1,\dots,u_M)\leq J(\sup_i|x_1-x_i|),\quad$    (decay)\\
\end{itemize}

where, for $x=(x_1,\dots,x_{M-1})\in(\rn)^{M-1}$, we use the notation $|x|_{\infty}=\sup_i|x_i|$. Note that the coercivity assumption depends on the geometry of the lattice. 
\\
\hspace*{0,5cm}
Under the assumptions (M1)-(M3) one can check that, up to minor changes, the proofs for pairwise interaction energies also work in this setting, so that the following theorems hold.

\begin{theorem}\label{main1multi}
Let $\Sigma$ be admissible and let $f^{\e}$ satisfy (M1)-(M3). For every sequence $\e_n\to 0^+$ there exists a subsequence $\e_{n_k}$ such that the functionals $F_{\e_{n_k},M}$ defined in \ref{multibody} $\Gamma$-converge with respect to the strong $\lud$-topology to a functional $F_M:\lud\rightarrow [0,+\infty]$ of the form

\begin{equation*}
F_M(u)=
\begin{cases}
\int_{S(u)}\phi^M_{\Sigma}(x,\nu_u)\,\mathrm{d}\mathcal{H}^{n-1} &\mbox{if $u\in BV(D,\{\pm 1\})$,} \\
+\infty &\mbox{otherwise.}
\end{cases}
\end{equation*}

Moreover a local version of the theorem holds: For all $u\in BV(D,\{\pm 1\})$ and all $A\in\Ard$

\begin{equation*}
\Gamma-\lim_k F_{\e_{n_k},M}(u,A)=\int_{S(u)\cap A}\phi^M_{\Sigma}(x,\nu_u)\,\mathrm{d}\mathcal{H}^{n-1}.
\end{equation*}
\end{theorem} 

\hspace*{0,5cm}
To state the convergence of boundary value problems, given a polyhedral function $u_{\varphi}$ on $A\in\Ard$ and a sequence $l_{\e}$ as in (\ref{slowboundary}), we define

\begin{equation}\label{boundarymulti}
F^{\varphi,l_{\e}}_{\e,M}(u,A)=
\begin{cases}
F_{\e,M}(u,A) &\mbox{if $u\in C_{\e}^{{\varphi},l_{\e}}(\Sigma,A)$,}\\
+\infty &\mbox{otherwise.}
\end{cases}
\end{equation}

\begin{theorem}\label{constrainedmulti}
Let $\Sigma$ be admissible and let $f^{\e}$ satisfy (M1)-(M3). For every sequence converging to $0$, let $\e_j$ and $\phi_{\Sigma}$ be as in Theorem \ref{main1multi}. Assume that the limit integrand $\phi_{\Sigma}$ is continuous on $D\times S^{n-1}$. For every Lipschitz set $A\subset\subset D$ the functionals $F^{\varphi,l_{\e_j}}_{\e_j,M}(\cdot,A)$ defined in (\ref{boundarymulti}) $\Gamma$-converge with respect to the strong $\lud$-topology to the functional $F^{\varphi}_M(\cdot,A):\lud\rightarrow [0,+\infty]$ defined by
\begin{equation*}
F^{\varphi}_M(u,A)=
\begin{cases}
\int_{S(u_{Q,\varphi})\cap \overline{A}}\phi^M_{\Sigma}(x,\nu_{u_{A,\varphi}})\,\mathrm{d}\mathcal{H}^{n-1} &\mbox{if $u\in BV(A,\{\pm 1\})$,}\\
+\infty &\mbox{otherwise.}
\end{cases}
\end{equation*}
\end{theorem}

\begin{theorem}\label{convofminimizersmulti}
Let $A\subset\subset D$ be a Lipschitz set. Under the assumptions of Theorem \ref{constrainedmulti}, the following holds:
\begin{enumerate}
\renewcommand{\labelenumi}{(\roman{enumi})}
\item
\begin{equation*}
\lim_j\left(\inf_{u\in BV(A,\{\pm 1\})}F^{\varphi,l_{\e_j}}_{\varepsilon_j,M}(u,A)\right)=\min_{u\in BV(A,\{\pm 1\})}F^{\varphi}_M(u,A).
\end{equation*}
\item Moreover, if $(u_j)_j$ is a converging sequence in $L^1(A,\{\pm 1\})$ such that
\begin{equation*}
F^{\varphi,l_{\e_j}}_{\varepsilon_j,M}(u_j,A)=\inf_{u\in BV(A,\{\pm 1\})}F^{\varphi,l_{\e_j}}_{\varepsilon_j,M}(u,A)+{ \scriptstyle \mathcal{O}} (1),
\end{equation*}
then its limit is a minimizer of $F^{\varphi}_M(\cdot,A)$.
\end{enumerate}
\end{theorem}

\hspace*{0,5cm}
For the stochastic homogenization we have to replace (\ref{periodicityassumpt}) by

\begin{equation}\label{multiperiodic}
f^{\e}(x_1,\dots,x_M,\cdot)=f(x_1-x_2,\dots,x_1-x_M,\cdot).
\end{equation}

Under this additional assumption also the analogue of Theorem \ref{mainthm2} holds.

\begin{theorem}\label{main2multi}
Let $\mathcal{L}$ be a stationary stochastic lattice and let $f^{\e}$ satisfy Hypothesis 1 with the additional structure of (\ref{multiperiodic}). For $\mathbb{P}$-almost every $\w$ and for all $\nu\in S^{n-1}$ there exists
\begin{equation*}
\phi^M_{\text{hom}}(\w;\nu):=\lim_{t\to +\infty}\frac{1}{t^{n-1}}\inf\left\{F_{1,M}(\w)(u,Q_{\nu}(0,t)):\; u\in C_1^{u_{0,\nu},l_{t^{-1}}}(\w,Q_{\nu}(0,t))\right\}.
\end{equation*}
The functionals $F_{\e,M}(\w)$ $\Gamma$-converge with respect to the $\lud$-topology to the functional $F_{\text{hom},M}(\w):\lud\rightarrow [0,+\infty]$ defined by
\begin{equation*}
F_{\text{hom},M}(\w)(u)=
\begin{cases}
\int_{S(u)} \phi^M_{\text{hom}}(\w;\nu_u)\,\mathrm{d}\mathcal{H}^{n-1} &\mbox{if $u\in BV(D,\{\pm 1\})$,}\\
+\infty &\mbox{otherwise.}
\end{cases}
\end{equation*}
If $\mathcal{L}$ is ergodic, then $\phi^M_{\text{hom}}(\cdot,\nu)$ is constant almost surely and is given by
\begin{equation*}
\phi^M_{\text{hom}}(\nu)=\lim_{t\to +\infty}\frac{1}{t^{n-1}}\int_{\Om}\inf\left\{F_{1,M}(\w)(u,Q_{\nu}(0,t)):\; u\in C_1^{u_{0,\nu},l_{t^{-1}}}(\w,Q_{\nu}(0,t))\right\}\,\mathrm{d}\mathbb{P}(\w).
\end{equation*}
\end{theorem}

\medskip

\noindent {\bf Acknowledgements}: The work of MR was supported by SFB / Transregio 109 'Discretization in Geometry and Dynamics'.
The authors thank A. Piatnitski for suggesting the key idea in the proof of Step $2$ in Theorem $5.5$.

\appendix
\section{}

\begin{lemma}\label{reshetnyaktype}
Let $A\subset\subset B$, $A,B\in\mathcal{A}^R(\R^n)$. Given $u\in BV(A,\{\pm 1\})$ let $u_{A,\varphi}$ be defined as in (\ref{tracefunction}). Then there exists a sequence $A_{\e}\subset\subset A$ of sets of finite perimeter (not depending on $B$) such that $u_{\e}:=u_{A_{\e},\varphi}$ converges strictly to $u_{A,\varphi}$ on $B$.  
\end{lemma}

\begin{proof}
As a special case of Proposition 4.1 in \cite{tschmidt}, applied to the $BV$-function $v:=u_{A,\varphi}-u_{\varphi}$, for every $\e>0$ we find an open set $A_{\e}$ of finite perimeter such that $A_{\e}\subset\subset A$, $|A\backslash A_{\e}|\leq \e$ and

\begin{equation}\label{stricttrace}
\int_{\partial^*A_{\e}}|v^-|\,\mathrm{d}\mathcal{H}^{n-1}\leq \int_{\partial A}|v^-|\,\mathrm{d}\mathcal{H}^{n-1}+\e,
\end{equation}

where $v^-$ denotes the interior trace. By refining in a trivial way the argument in \cite{tschmidt} the sets $A_{\e}$ can be constructed in a way such that, for all $\delta>0$ there exists $\e_0>0$ such that for all $\e<\e_0$

\begin{equation}\label{fillset}
\{x\in A:\;\dist(x,\partial A)>\delta\}\subset A_{\e}.
\end{equation}

\hspace*{0,5cm}
We show that the sets $A_{\e}$ fulfill the required properties. It is easy to see that $u_{\e}$ converges to $u_{A,\varphi}$ in $L^1(B)$. By lower semicontinuity of the total variation it is enough to show that 

\begin{equation}\label{strictproperty}
\limsup_{\e\to 0}|Du_{\e}|(B)\leq |Du_{A,\varphi}|(B).
\end{equation}

By definition we have $|Du_{\e}|(B\backslash\overline{A})=|Du_{A,\varphi}|(B\backslash\overline{A})$ so that we can reduce the analysis to $\overline{A}$. Note that

\begin{equation*}
Du_{\e}=\mathds{1}_{A_{\e}^{(1)}}Du+\mathds{1}_{A_{\e}^{(0)}}Du_{\varphi}+(u_{\varphi}^+-u^-)\cdot\nu\,\mathcal{H}^{n-1}_{|\partial^*A_{\e}},
\end{equation*}

where $A_{\e}^{(0)},\,A_{\e}^{(1)}$ denote the points with density $0$ respectively $1$ with respect to $A_{\e}$. Since $A_{\e}\subset\subset A$ and $A_{\e}$ is open we infer
\begin{align*}
|Du_{\e}|(\overline{A})&\leq|Du|(A)+|Du_{\varphi}|(\overline{A}\backslash A_{\e})+\int_{\partial^*A_{\e}}|u_{\varphi}^+-u^-|\,\mathrm{d}\mathcal{H}^{n-1}
\\
&\leq |Du|(A)+|Du_{\varphi}|(\overline{A}\backslash A_{\e})+\int_{\partial^* A_{\e}}|u_{\varphi}^+-u_{\varphi}^-|\,\mathrm{d}\mathcal{H}^{n-1}+\int_{\partial^* A_{\e}}|u_{\varphi}^{-}-u^-|\,\mathrm{d}\mathcal{H}^{n-1}.
\end{align*}

By assumption on $u_{\varphi}$ we have $\mathcal{H}^{n-1}(Su_{\varphi}\cap\partial A)=0$, so that by (\ref{fillset}) the second and the third term vanish when $\e\to 0$. For the third one we use (\ref{stricttrace}) to infer

\begin{align*}
\limsup_{\e\to 0}|Du_{\e}|(\overline{A}) &\leq |Du|(A)+\int_{\partial A}|u^{-}-u_{\varphi}^-|\,\mathrm{d}\mathcal{H}^{n-1}
\\
&=|Du|(A)+\int_{\partial A}|u^{-}-u_{\varphi}^+|\,\mathrm{d}\mathcal{H}^{n-1}=|Du_{A,\varphi}|(\overline{A}),
\end{align*}

where we have used that on $\partial A$ the inner and outer traces of $u_{\varphi}$ agree.

\end{proof}

\begin{lemma}
For every $L\in\mathbb{N}, I\in\mathcal{I}$ and every rational direction $\nu\in S^{n-1}$ the function $\tilde{\mu}^L_{\nu}(I)$ defined in (\ref{process}) is $\mathcal{F}$-measurable.
\end{lemma}

\begin{proof}
For $0<r<R$, we denote by $\Sigma_{r,R}$ the space of all admissible lattices with corresponding constants $r,R$ (in the sense of Definition \ref{defadmissible}). Since $\mathcal{F}$ is a complete $\sigma$-algebra, we can assume that $\Lw\in\Sigma_{r,R}$ for all $\w\in\Om$. Given $i,j\in\mathbb{N}$, we first prove that the set of all $x\in\Sigma_{r,R}$ such that $x_{i}$ and $x_{j}$ are nearest neighbours is measurable. Note that $x_{i}$ and $x_{j}$ are nearest neighbours if and only if

\begin{equation*}
\exists y\in\mathbb{R}^n:\quad |y-x_{i}|=|y-x_{j}|<|y-x_{k}|\quad\forall k\neq i,j.
\end{equation*}

Let us take a countable collection $\{B_n\}_n$ of connected sets that form a basis of the norm topology in $\mathbb{R}^n$. Using the fact that $x\in\Sigma_{r,R}$ on the one hand and the intermediate value theorem on the other hand one can check that the above characterization is equivalent to
\begin{align*}
x\in\bigcup_{n\in\mathbb{N}}&\Big(\{y\in\Sigma_{r,R}:\;\sup_{v\in B_n}|y_{i}-v|-|y_{j}-v|\geq 0,\,\inf_{v\in B_n}|y_{i}-v|-|y_{j}-v|\leq 0\}
\\
&\cap \bigcap_{k\in\mathbb{N}\backslash\{i,j\}}\{y\in\Sigma_{r,R}:\;\sup_{v\in B_n}|y_{i}-v|-|y_{k}-v|<0\}\Big)
\end{align*}
The last set is a countable union of measurable sets, hence measurable. It remains to show that the infimum in the definition can be taken over a countable set. The discrete constraints near the boundary can be replaced by a measurable penalty term of the form

\begin{equation*}
\sum_{i\geq 1}C\cdot (v(\Lw_{i})-u_{0,\nu}(\Lw_{i}))\cdot\mathds{1}_{\{y\in\rn:\,\dist(y,\partial I)\leq L+r\}}(\Lw_{i}),
\end{equation*}
where $C$ is large enough to dominate the right-hand side of (\ref{minestimate}). Finally we minimize over the first $M$ coordinates of vectors in $\{\pm 1\}^{\NN}$ (the others being constantly $1$) and then let $M\to+\infty$ to see that $\tilde{\mu}^L_{\nu}(I)$ can be written as the pointwise limit of measurable functions. Note that the limit exists since only finitely many points of the lattice are contained in $I$, so that the minimization process is finally constant.
\end{proof}


\begin{thebibliography}{99}

\bibitem{ergodic}
U.~Akcoglu, M.A.~Krengel.
\newblock Ergodic theorems for superadditive processes.
\newblock {\em J. Reine Angew. Math.}, 323:53--67, 1981.

\bibitem{ABC}
R.~Alicandro, A.~Braides, and M.~Cicalese.
\newblock Phase and anti-phase boundaries in binary discrete systems: a
  variational viewpoint.
\newblock {\em Netw. Heterog. Media}, 1(1):85--107, 2006.

\bibitem{ACXY}
R.~Alicandro and M.~Cicalese.
\newblock Variational analysis of the asymptotics of the ${\MakeUppercase
  {xy}}$ model.
\newblock {\em Arch. Ration. Mech. Anal.}, 192(3):501--36, 2009.

\bibitem{ACG2}
R.~Alicandro, M.~Cicalese, and A.~Gloria.
\newblock Integral representation results for energies defined on stochastic
  lattices and application to nonlinear elasticity.
\newblock {\em Arch. Ration. Mech. Anal.}, 200(3):881--943, 2011.

\bibitem{ACP}
R.~Alicandro, M.~Cicalese, and M.~Ponsiglione.
\newblock Variational equivalence between {G}inzburg-{L}andau, ${\MakeUppercase
  {xy}}$ spin systems and screw dislocation energies.
\newblock {\em Indiana Univ. Math. J.}, 60(1):171--208, 2011.

\bibitem{ACS}
R.~Alicandro, M.~Cicalese, and L.~Sigalotti.
\newblock Phase transitions in presence of surfactants: from discrete to
  continuum.
\newblock {\em Interfaces Free Bound.}, 14(1):65--103, 2012.

\bibitem{AlGe14}
R.~Alicandro and M.~S. Gelli.
\newblock Local and non local continuum limits of {\MakeUppercase i}sing type
  energies for spin systems.
\newblock {\em (preprint http://cvgmt.sns.it/paper/2496/)}, submitted 2014.

\bibitem{AmBrII}
L.~Ambrosio and A.~Braides.
\newblock Functionals defined on partitions of sets of finite perimeter, {II}:
  semicontinuity, relaxation and homogenization.
\newblock {\em J. Math. Pures. Appl.}, 69:307--333, 1990.

\bibitem{AFP}
L.~Ambrosio, N.~Fusco, and D.~Pallara.
\newblock {\em Functions of bounded variation and free discontinuity problems}.
\newblock Oxford Mathematical Monographs. The Clarendon Press Oxford University
  Press, New York, 2000.

\bibitem{BLB}
X.~Blanc and C.~Le~Bris.
\newblock D\'efinition d'\'energies d'interfaces \`a partir de mod\`eles
  atomiques.
\newblock {\em C. R. Math. Acad. Sci. Par.}, 340:535--540, 2005.

\bibitem{BLBL}
X.~Blanc, C.~Le~Bris, and P.~L. Lions.
\newblock Form molecular models to continuum mechanics.
\newblock {\em Arch. Ration. Mech. Anal.}, 164:341--381, 2005.

\bibitem{BFLM}
G.~Bouchitt\'e, I.~Fonseca, G.~Leoni, and L.~Mascarenhas.
\newblock A global method for relaxation in ${\MakeUppercase w}^{1,p}$ and in
  ${\MakeUppercase {sbv}}_p$.
\newblock {\em Arch. Ration. Mech. Anal.}, 165(3):187--242, 2002.

\bibitem{GCB}
A.~Braides.
\newblock {\em {$\Gamma$}-convergence for beginners}, volume~22 of {\em Oxford
  Lecture Series in Mathematics and its Applications}.
\newblock Oxford University Press, Oxford, 2002.

\bibitem{BCS}
A.~Braides, M.~Cicalese, and F.~Solombrino.
\newblock ${\MakeUppercase q}$-tensor continuum energies as limits of
  head-to-tail symmetric spin systems.
\newblock {\em preprint}, 2013.

\bibitem{brde}
A.~Braides and A.~Defranceschi.
\newblock {\em Homogenization of Multiple Integrals}.
\newblock Oxford University Press, Oxford, 1998.

\bibitem{BP}
A.~Braides and A.~Piatnitski.
\newblock Homogenization of surface and length energies for spin systems.
\newblock {\em J.Funct. Anal.}, 264:1296--1328, 2013.

\bibitem{BrSo}
A.~Braides and M.~Solci.
\newblock Interfacial energies on penrose lattices.
\newblock {\em Math. Models Methods Appl. Sci. (M3AS)}, 21:1193--1210, 2011.

\bibitem{CdlL}
L.A. Caffarelli and R.~de~la Llave.
\newblock Planelike minimizers in periodic media.
\newblock {\em Commun. Pure and Appl. Math.}, 54(12):1403--1441, 2001.

\bibitem{CDSZ}
M.~Cicalese, A.~DeSimone, and C.~Zeppieri.
\newblock Discrete-to-continuum limits for strain-alignment-coupled systems:
  magnetostrictive solids, ferroelectric crystals and nematic elastomers.
\newblock {\em Netw. Heterog. Media}, 4(4):667--708, 2009.

\bibitem{CS}
M.~Cicalese and F.~Solombrino.
\newblock Frustrated ferromagnetic spin chains: a variational approach to
  chirality transitions.
\newblock {\em preprint}, 2013.

\bibitem{DM}
G.~Dal~Maso.
\newblock {\em An introduction to {$\Gamma$}-convergence}.
\newblock Progress in Nonlinear Differential Equations and their Applications,
  vol. 8. Birkh{\"a}user Boston Inc., Boston, MA, 1993.

\bibitem{DMM}
G.~Dal~Maso and L.~Modica.
\newblock Nonlinear stochastic homogenization and ergodic theory.
\newblock {\em J. Reine Angew. Math.}, 368:28--42, 1986.

\bibitem{doktor}
P.~Doktor.
\newblock Approximation of domains with lipschitzian boundary.
\newblock {\em \v{C}asopis pro p\v{e}stov\'{a}ni matematiky}, 101(3):237--255,
  1976.

\bibitem{modernmethods}
I.~Fonseca and G.~Leoni.
\newblock {\em Modern Methods in the Calculus of Variations: $L^p$ Spaces}.
\newblock Springer, New York, 2010.

\bibitem{glpe}
A.~Gloria and M.~D. Penrose.
\newblock Random parking, euclidean functionals, and rubber elasticity.
\newblock {\em Communications in Mathematical Physics}, 321(1):1--31, 2013.

\bibitem{Ponsiglione}
M.~Ponsiglione.
\newblock Elastic energy stored in a crystal induced by screw dislocations:
  from discrete to continuous.
\newblock {\em SIAM J. Math. Anal.}, 39(2):449--469, 2007.

\bibitem{Presutti}
E.~Presutti.
\newblock {\em Scaling limits in statistical mechanics and microstructures in
  continuum mechanics}.
\newblock Theoretical and Mathematical Physics. Springer, Berlin, 2009.

\bibitem{ruelle}
D~Ruelle.
\newblock {\em Statistical Mechanics. Rigorous results}.
\newblock River Edge, NJ: World Scientific, Reprint of the 1989 edition.

\bibitem{tschmidt}
T.~Schmidt.
\newblock Strict interior approximation of sets of finite perimeter and
  functions of bounded variation.
\newblock {\em Proc. Am. Math. Soc.} 143(5),:2069--2084, 2015.

\bibitem{vollath2013}
D.~Vollath.
\newblock {\em Nanoparticles-Nanocomposites Nanomaterials: An Introduction for
  Beginners}.
\newblock Wiley, New York, 2013.

\bibitem{lin}
C.~L. Zihwei.
\newblock Extending an orthonormal rational set of vectors into an orthonormal
  rational basis.
\newblock www.math.uchicago.edu/~may/VIGRE/VIGRE2006/PAPERS/Lin.pdf.

\end{thebibliography}
\end{document}